\documentclass[11pt]{article}
\usepackage{amsmath,amssymb,latexsym,amsfonts,amsthm}
\usepackage{bbm}
\usepackage{multirow}	
\usepackage{enumerate}
\usepackage[dvips]{graphics}
\usepackage{graphicx}
\usepackage{epstopdf}
\usepackage{color}
\usepackage{booktabs}
\usepackage[tight]{units} 
\numberwithin{equation}{section} 

\newtheorem{theorem}{Theorem}[section]

\newtheorem{lemma}{Lemma}[section]
\newtheorem{corollary}{Corollary}[section]

\newtheorem{problem}{Problem}[section]
\newcommand{\curl}{\mathop{\mathbf{curl}}\nolimits}
\newcommand{\dive}{\mathop{\mathrm{div}}\nolimits}
\newcommand{\dif}{\overline{\partial}}
\newcommand{\grad}{\mathop{\mathbf{\nabla}}\nolimits}

\newcommand{\norm}[1]{\left\lVert#1\right\rVert}
\newcommand{\set}[1]{\left\{#1\right\}}
\newcommand{\produ}[1]{\left\langle#1\right\rangle}

\newcommand{\M}{M}
\newcommand{\X}{X}
\newcommand{\Y}{Y}
\newcommand{\V}{V}
\newcommand{\R}{\mathbb{R}}
\newcommand{\oc}{\Omega_{{}_{\mathrm C}\!}}
\newcommand{\od}{\Omega_{{}_{\mathrm D}\!}}

\newcommand{\adhod}{\overline{\Omega}_{\mathrm{d}}}

\newcommand{\xho}{\X_h}
\newcommand{\moh}{M_h}

\renewcommand{\H}{\mathrm{H}}
\renewcommand{\L}{\mathrm{L}}
\newcommand{\hcurlo}{\mathbf{H}(\curl;\Omega)}
\newcommand{\hocurlod}{\mathbf{H}_0(\curl;\od)}
\newcommand{\hcurloc}{\mathbf{H}(\curl;\oc)}
\newcommand{\hocurlo}{\mathbf{H}_0(\curl;\Omega)}

\newcommand{\hcurlod}{\mathbf{H}(\curl;\od)}

\newcommand{\hunod}{\H^1(\od)}
\newcommand{\ldos}{\L^2}
\newcommand{\ldoso}{\L^2(\Omega)}
\newcommand{\ldosod}{\L^2(\od)}
\newcommand{\ldoshocurlo}{\L^2(0,T;\hocurlo)}

\newcommand{\ldosoc}{\L^2(\oc)}

\newcommand{\cero}{\boldsymbol{0}}
\newcommand{\ds}{\,ds}
\newcommand{\dt}{\,dt}
\newcommand{\bn}{\mathbf{n}}
\newcommand{\xn}{\mathbf{x}}

\newcommand{\zn}{\mathbf{z}}
\newcommand{\un}{\mathbf{u}}
\newcommand{\fn}{\mathbf{f}}
\newcommand{\vn}{\mathbf{v}}
\newcommand{\E}{\mathcal{E}}

\newcommand{\wn}{\mathbf{w}}

\newcommand{\Hn}{\mathbf{H}}

\newcommand{\vnc}{\vn_{\mathrm{c}}}
\newcommand{\pih}{\Pi_h}

\newcommand{\lodoso}{\L_0^2(\Omega)}
\newcommand{\houno}{\H_0^1(\Omega)}
\newcommand{\poli}{\mathbb{P}}
\newcommand{\unhnm}{\un_h^{n-1}}
\newcommand{\Phnm}{P_h^{n-1}}
\newcommand{\hounoO}{\H_0^1(\Omega)}
\newcommand{\abs}[1]{\left\lvert#1\right\rvert}
\newcommand{\ldoshounod}{\L^2(0,T;\hounoO^d)}

\newcommand{\vc}{\vn_{\!\mathrm{c}}}

\newcommand{\nd}{\mathcal{N}_1}
\newcommand{\unhn}{\un_h^n}
\newcommand{\Phn}{P_h^n}

\newcommand{\Gc}{\Gamma_{{}_{\!\mathrm C}\!}}
\newcommand{\Gd}{\Gamma_{{}_{\!\mathrm D}\!}}
\newcommand{\GI}{\Gamma_{{}_{\!\mathrm I}\!}}
\newcommand{\GE}{\Gamma_{{}_{\!\mathrm E}\!}}
\newcommand{\GJ}{\Gamma_{{}_{\!\mathrm J}\!}}

\newcommand{\nn}{\boldsymbol{n}}
\newcommand{\HGccurloc}{\mathbf{H}_{\Gc}(\curl;\oc)}

\newcommand{\HGIcurlod}{\mathbf{H}_{\GI}(\curl;\od)}
\newcommand{\HHGIcurlod}{\widehat{\mathbf{H}}_{\GI}(\curl;\od)}
\newcommand{\HGIcurlood}{{\mathbf{H}}_{\GI}(\curl^{\cero};\od)}
\newcommand{\HGddivood}{{\mathbf{H}}_{\Gd}(\dive^{0}_{\varepsilon};\od)}
\newcommand{\hdivod}{\mathrm{H}(\dive,\od)}
\newcommand{\HHdivod}{\widehat{\mathrm{H}}(\dive;\od)}
\newcommand{\LdosGd}{\mathrm{L}^2(\Gd)}
\newcommand{\HH}{\mathbb{H}}
\newcommand{\wnd}{\wn_{{}_{\mathrm D}\!}}
\newcommand{\vnd}{\vn_{{}_{\mathrm D}\!}}
\newcommand{\znd}{\zn_{{}_{\mathrm D}\!}}
\newcommand{\gtc}{\boldsymbol{\gamma}_\tau^{{}^{\mathrm C}\!}}
\newcommand{\gtd}{\boldsymbol{\gamma}_\tau^{{}^{\mathrm D}\!}}
\newcommand{\rangoGtoc}{\mathrm{H}^{-1/2}(\dive_{\tau};\partial\oc)}
\newcommand{\rangoGtod}{\mathrm{H}^{-1/2}(\dive_{\tau};\partial\od)}
\newcommand{\etan}{\boldsymbol{\eta}}
\newcommand{\LL}{\mathcal{L}}
\newcommand{\HGdcurloodGI}{{\mathbf{H}}_{\GI}(\dive^{0};\od)}
\newcommand{\hunodGI}{\H_{\GI}^1(\od)}
\newcommand{\hdive}{\mathbf{H}(\dive;\Omega)}
\newcommand{\hodive}{\mathbf{H}_0(\dive;\Omega)}
\newcommand{\HGdcurlood}{{\mathbf{H}}_{\Gd}(\curl;\od)}

\newcommand{\OS}{\Omega_{\mathrm s}}
\newcommand{\bJ}{\mathbf J}
\newcommand{\bE}{\mathbf E}
\newcommand{\bx}{\mathbf x}
\newcommand{\bH}{\mathbf H}

\newtheorem{remark}[theorem]{Remark}

\def\thanksramiro{Universidad del Cauca, Popay\'an, Colombia, 
 email: {\tt rmacevedo@unicauca.edu.co},}
                  
\def\thankschristian{Universidad del Cauca, Popay\'an, Colombia,
                  email: {\tt christiancamilo@unicauca.edu.co},
                }  
\def\thanksbibiana{Universidad Nacional sede Medell\'in, Colombia,
                  email: {\tt blopezr@unal.edu.co}.
                }  

\begin{document}
\title{
Fully discrete finite element approximation for a family of degenerate 
parabolic mixed equations}


\author{
{\sc Ramiro Acevedo}\thanks{\thanksramiro}\quad
{\sc Christian G\'omez}\thanks{\thankschristian}\quad  
{\sc Bibiana L\'opez-Rodr\'{\i}guez}\thanks{\thanksbibiana}
}

\date{Received: date / Accepted: date}

\maketitle 

\begin{abstract}
{The aim of this work is to show an abstract framework to analyze 
the numerical approximation for a family of linear degenerate parabolic mixed equations 
by using a finite element method in space and a Backward-Euler scheme in time. 
We consider sufficient conditions to prove that the fully-discrete problem has a unique solution 
and prove quasi-optimal error estimates for the approximation.
Furthermore, we show that mixed finite element formulations arising from dynamics fluids (\textit{time-dependent Stokes problem)}  and from electromagnetic applications \textit{(eddy current models)}, can be analyzed as applications of the developed theory. Finally, we include numerical tests to illustrate the performance
of the method and confirm the theoretical results.}

\end{abstract}
\section{Introduction}\label{sec;introduction}
The classical (non-degenerate) mixed parabolic equations, 
mainly inspired on the Stokes Problem, have been widely studied. 
For instance, 
in \cite{BR} the authors have introduced some abstract framework for that kind of problems.  
However, mixed formulations arising from electromagnetic problems (see, for instance \cite{A1,A2,blrs:13}), 
can not fit in that aforementioned theory, because in these cases the first term inside of the time-derivative is not an inner product,
which implies that the resulting problem is degenerate.

A degenerate parabolic mixed problem 
consists in finding $u\in L^2(0,T;\X)$ and $\lambda\in H^1(0,T;\M)$ such that: 
\begin{equation}
\begin{split}
 &\frac{d}{dt}\left[\langle Ru(t),v\rangle_{\Y}+ b(v,\lambda(t))\right]+\langle Au(t),v\rangle_{\X}
 =\langle f(t),v\rangle_{\X}\qquad \forall v\in\X \quad\text{in }\mathcal{D}'(0,T),\\
\label{PC}
& b(u(t),\mu)=\langle g(t),\mu\rangle_{\M}\qquad \forall\mu\in\M,
\end{split}
\end{equation}
where $\X$, $\Y$ and $\M$ are real Hilbert spaces with the imbedding ${\X\subseteq \Y}$ is continuous and dense, 
$R:Y\to Y'$, $A:X\to X'$ are linear and bounded operators, $b:\X\times\M\to\R$ is a bounded bilinear form,
$f\in L^2(0,T;\X')$ and $g\in L^2(0,T;M')$. 
This kind of problems appears in several applications, for instance, to the approximation of the heat equation by means 
of Raviart-Thomas (RT) elements (see \cite{JT}), to the fluid dynamic equations (see \cite{BR} and \cite{BG}), 
and to electromagnetic applications (see \cite{A1,A2,blrs:13,MS1}). Sufficient conditions to the well-posedness 
of Problem \eqref{PC}  with an appropriate initial condition, are given in the recent paper \cite{AGL}. 
In that work,  the authors combine the linear degenerate parabolic equation theory with the classical Babu$\check{\text{s}}$ka-Brezzi 
Theory to prove the existence and uniqueness of solution, by assuming some reasonable conditions inspired by the application problems. 

%

The purpose of this paper is to provide the analysis for a fully-discrete approximation of the abstract Problem~\eqref{PC}. 
This approximation is obtained by using the finite element method in space with a backward Euler in time. In order to develop the analysis,
it is necessary to assume the conditions considered in \cite{AGL} to ensure the well-posedness of the continuous problem~\eqref{PC}.
Furthermore, to obtain the existence and uniqueness for the fully discrete solutions, the bilinear form induced for the operator $A$ 
have to satisfy a discrete G\r{a}rding-type inequality 
in the discrete kernel of the bilinear form $b$ and this bilinear form $b$ must satisfy the discrete inf-sup condition. 
The discrete inf-sup condition of $b$ plays an important role to adapt the techniques from the error analysis for 
finite element approximations for classical  parabolic problems. In fact, 
this discrete inf-sup condition allows to define some projection operator to the orthogonal of the discrete kernel of $b$, 
which it is necessary to obtain the suitable split of the error to prove quasi-optimal error estimates for the approximation of the main variable of the problem.
Moreover, by using again the discrete inf-sup condition, we can obtain the intermediate term to get the corresponding split for 
the approximation error of the Lagrange Multiplier to show the theoretical convergence of the method.  

About the applications for the theory of the fully-discrete approximation of Problem~\eqref{PC}, we present some problems 
that arise from dynamic fluids and electromagnetic models. Firstly, 
we give the convergence analysis of the approximation of the time-dependent Stokes problem. 
Thus, it can be inferred that the fully-discrete analysis for the non-degenerate mixed problems is a particular case 
of the theory studied in this work. 
Many of the real degenerate parabolic problems mainly come from electromagnetic applications, because the existence of two kind of materials 
(conductors and insulators) lead to the problem has a degenerate character. In fact 
we show two applications for an electromagnetic problem called the eddy current model. The formulations studied here are based in 
a time primitive of the electric field and they were studied respectively in \cite{A1} (for the case of internal density current source) and in \cite{blrs:13}
(for the case of density sources with current excitations). 
The use of the time-primitive of the electric field as main unknown is called 
\textit{modified magnetic vector potential} in the electrical engineering 
literature (see, for instance \cite{EJ}). Additionally, we perform some numerical results
that corroborate the convergence order given for the theory for the model studied in \cite{A1} (see Section \ref{IC}),
since their authors did not present numerical simulation in that work.



The outline of the paper is as follows: In Section \ref{continuo} we recall the main results given in \cite{AGL} about the well-posedness 
of the abstract problem~\eqref{PC}
and the corresponding analysis for its fully-discrete approximation scheme
is presented in Section~\ref{discreto}. The results concerning for error estimates of the fully-discrete approximation of the problem 
are shown in Section~\ref{estimates}. The applications of the theory to the time-dependent Stokes problem and to 
the eddy current model are studied in Section~\ref{aplicaciones}, where we
use the developed abstract theory to
deduce the well-posedness of the discrete problems and the theoretical convergence for their approximations.
Furthermore, we show some numerical results for the first of the eddy current models that confirm the expected convergence 
of the method according to the theory.

\section{An abstract degenerate mixed parabolic problem}\label{continuo}
Let $\X$ and $\Y$ be two real Hilbert spaces such that $\X$ is contained in $\Y$ with a continuous and dense imbedding.
Furthermore, let $\M$ be real reflexive Banach space. Then,  we consider 
the continuous operators $R: \Y\to \Y'$ , $A:\X\to \X'$ and  $b:\X\times \M\to \R$ be a continuous bilinear form.
Let $V$ be the kernel of the bilinear form $b$, i.e.,
\begin{equation}\nonumber
V:=\left\{v\in \X: b(v,\mu)=0\ \ \forall\mu\in \M\right\},
\end{equation}
and denote by $W$ its clausure with respect to the $\Y$-norm, i.e.,
\begin{equation}\nonumber
W:=\overline{V}^{\| \cdot \|_{\Y}}.
\end{equation}
We consider now the following abstract problem

Given $u_0\in\Y$, $f\in\L^2(0,T;\X')$ and $g\in\L^2(0,T;\M')$, the continuous problem is 
\begin{problem}\label{PC1}
 Find $u\in\L^2(0,T;\X)$ and $\lambda\in\L^2(0,T;\M)$ satisfying the following equations:
\begin{align*}
& \frac{d}{dt}\left[\langle Ru(t),v\rangle_{\Y}+ b(v,\lambda(t))\right]+\langle Au(t),v\rangle_{\X}
=\langle f(t),v\rangle_{\X}&& \forall v\in\X \quad\text{in }\mathcal{D}'(0,T),\\
& b(u(t),\mu)=\langle g(t),\mu\rangle_{\M}&& \forall\mu\in\M,\\
& \langle Ru(0),v\rangle_{\Y}=\langle Ru_0,v\rangle_{\Y}&& \forall v\in \Y.
\end{align*}
\end{problem}
In this order, the hypotheses  that guarantee is a well-posed problem are given by
\begin{itemize}
\item[H1.] The bilinear form $b$ satisfies a continuous \textit{inf-sup} condition, i.e., there exists $\beta>0$ such that
\begin{equation}\nonumber
\underset{v\in\X}{\sup}\frac{b(v,\mu)}{\|v\|_{\X}}\geq \beta \|\mu\|_{\M}\quad\forall \mu\in\M.
\end{equation}
\item[H2.] $R$ is self-adjoint and monotone on $V$, i.e., 
\begin{equation}\nonumber
\langle Rv,w \rangle_{\Y}=\langle Rw,v \rangle_{\Y},\qquad \langle  Rv,v \rangle_{\Y}\geq 0\ \qquad \forall v,w\in V.
\end{equation}
\item[H3.] The operator $A$ is self-adjoint on $V$, i.e.,
\begin{equation}\nonumber
\langle Av,w \rangle_{\X}=\langle Aw,v \rangle_{\X}\ \qquad \forall v,w\in V
\end{equation}
\item[H4.] There exist $\gamma>0$ and  $\alpha>0$ such that
\begin{equation*}\label{garding}
\langle Av,v \rangle_{\X}+\gamma\langle Rv,v \rangle_{\Y}\geq \alpha \|v\|^2_{\X}\ \qquad \forall v\in V.
\end{equation*}
\item[H5.] The initial data $u_0$ belongs to $W$.
\item[H6.] The data function $g$ belongs to $\H^1(0,T;\M')$.
\end{itemize}
\begin{theorem}\label{princonti}
Let us assume that assumptions {H1}--{H6} hold true. Then the Problem \ref{PC1} has a unique solution 
$(u,\lambda)\in\L^2(0,T;\X)\times\H^1(0,T;\M))$ and there exists a constant $C>0$ such that
\begin{equation}\nonumber
\|u\|_{\L^2(0,T;\X)} + \|\lambda\|_{\L^2(0,T;\M)} 
\leq C \left\{
\|f\|_{\L^2(0,T;\X')} + \|g\|_{\H^1(0,T;\M')} + \|u_0\|_{\Y}
\right\}.
\end{equation}
Moreover,  $\lambda(0)=0$.
\end{theorem}
\begin{proof}
(see \cite[Theorem 2.1]{AGL})
\end{proof}
In the following section we present the fully discrete analysis for the Problem  \ref{PC1}.

\section{Fully-discrete approximation for degenerate mixed parabolic problem}\label{discreto}
Let $\{\X_h\}_{h>0}$ and $\{\M_h\}_{h>0}$ be sequences of finite-dimensional subspaces of $\X$ and $\M$, respectively, and let
$\{t_n: =n\Delta t: \ n=0,...,N\}$ be a uniform partition of $[0,T]$ with a step size $\Delta t := T/N$. For any finite sequence $\{\theta^n: \ n=0,...,N\}$ 
we denote
\[
\dif\theta^n:=\frac{\theta^n-\theta^{n-1}}{\Delta t},\qquad n=1,\dots,N.
\]
The fully-discrete approximation of the Problem \ref{PC1}  reads as follows:

\begin{problem} \label{PD1} Find $u_h^n\in\X_h$, $\lambda_h^n\in\M_h$, $n=1,\ldots,N$, such that
\begin{align*}
& \langle R\dif u_h^n,v\rangle_{\Y}+ b(v,\dif\lambda_h^n)+\langle Au_h^n,v\rangle_{\X}
=\langle f(t_n),v\rangle_{\X}&& \forall v\in\X_h,\\
& b(u_h^n,\mu)=\langle g(t_n),\mu\rangle_{\M}&& \forall\mu\in\M_h,\\
& u_h^0= u_{0,h},\\
&\lambda_h^0= 0.
\end{align*}
\end{problem}
This scheme is obtained by using a backward Euler discrete approximation for the time-derivatives. Furthermore, the 
third equation the Problem \ref{PD1} includes a suitable approximation $u_{0,h}$ of the initial data $u_0$ to obtain the convergence of the scheme (see Subsection \ref{estimates} below).


On the other hand, to obtain the well-posedness of the Problem \ref{PD1}, we first notice that by rewriting equations, the solution  
$(u_h^n, \lambda_h^n) \in\X_h\times\M_h$ at each time step have to satisfy the following classical mixed problem:
\begin{align*}
\mathcal{A}(u_h^n,v) + b(v,\lambda_h^n) &= F_n(v)&& \forall v\in X_h,\\
b(u_h^n,\mu)&=\langle g(t_n),\mu\rangle_{\M}&& \forall \mu\in \M_h,
\end{align*}
where
\begin{align*}
\mathcal{A}(w,v)&:=\langle R w,v\rangle_{\Y}+\Delta t \ \langle Aw,v\rangle_{\X},\\
F_n(v)&:=\Delta t\ \langle f(t_n),v\rangle_{\X}+\langle R u_h^{n-1},v\rangle_{\Y}+b(v,\lambda_h^{n-1}).
\end{align*}
Hence, the existence and uniqueness of solution of the Problem \ref{PD1} is obtained by 
assuming the following conditions and using the classical Babuska-Brezzi Theory:
\begin{itemize}
\item[H7.] There exist $\xi_h>0 \textrm{ and }\ \alpha_h>0$ such that
\begin{equation}\nonumber
\langle Av,v \rangle_{\X}+\xi_h \langle Rv,v \rangle_{\Y}\geq \alpha_h \|v\|^2_{\X}\ \qquad \forall v\in V_h,
\end{equation}
where $V_h$ denotes the discrete kernel of $b$ in $\X_h$, i.e.,
\begin{equation*}
V_h:=\{v\in \X_h\,:\,b(v,\mu)=0\ \ \forall\mu\in\M_h\}.
\end{equation*}
\item[H8.] The bilinear form $b: \X_h\times\M_h\to \R$ is bounded and it satisfies the discrete \textit{inf-sup} condition, i.e., there exists $\beta_h>0$ such that
\begin{equation*}
\underset{v\in\X_h}{\sup}\frac{b(v,q)}{\|v\|_{\X}}\geq \beta_h \|q\|_{\M}\quad\forall q\in\M_h.
\end{equation*}
\end{itemize}

\section{Error estimates for the fully-discrete approximation}\label{estimates}
\subsection{Error estimates for main variable $u$}
We need to introduce the operators $\tilde{B}_h:\X\to \M_h'$ and  $B_h:\X_h\to \M_h'$ defined as follows 
\begin{align*}
& \langle\tilde{B}_hv,\mu\rangle_{\M}:=b(v,\mu) \quad\forall v\in\X,\quad \forall \mu\in\M_h,\\
& \langle B_hv,\mu\rangle_{\M}:=b(v,\mu) \quad\forall v\in\X_h,\quad \forall \mu\in\M_h.
\end{align*}
Now, we notice that if we consider the operator $\pih: \X\to \X_h$ characterized by
\begin{equation*}
\pih w\in \X_h:\qquad
(\pih w,z)_{\X}=(w,z)_{\X}\qquad\forall z\in X_h,
\end{equation*}
then there exists $C>0$ independent on $h$ satisfying
\begin{equation}\label{infimoproyector}
\|w-\Pi_h w\|_{\X}\leq C \inf_{z\in \X_h}\|w-z\|_{\X}.
\end{equation}
From the second and third equation of the Problem \ref{PD1}, we obtain
\[
b(u(t_n)-u_h^n,\mu)=0\qquad \forall \mu\in\M_h,
\]
and therefore
\begin{equation*}
\label{DC}
b(u(t_n)-\Pi_h u(t_n),\mu)=b(u_h^n-\Pi_h u(t_n),\mu)\qquad \forall \mu\in\M_h.
\end{equation*}

On the other hand, since $B_h$ satisfies the discrete \textit{inf-sup} condition, for all  $w\in X$ there exists a unique 
$\tilde{\mathcal{P}}_hw\in V_h^\perp\subset X_h$ such that
\[
B_h\tilde{\mathcal{P}}_hw=\tilde{B}_h(w-\Pi_h w), 
\]
or equivalently, 
\[
b(\tilde{\mathcal{P}}_hw,\mu)=b(w-\Pi_h w,\mu)\quad \forall\mu\in M_h.
\]
Furthermore, for all $w\in X$, we have
\begin{align*}
\norm{\tilde{\mathcal{P}}_hw}_X&\leq\frac{1}{\beta_h}\norm{B_h\tilde{\mathcal{P}}_hw}_{M'} 
\leq\frac{1}{\beta_h}\ \underset{\mu\in\M_h}{\sup} \frac{b(\tilde{\mathcal{P}}_hw,\mu)}{\|\mu\|_{\M}}
=\frac{1}{\beta_h}\ \underset{\mu\in\M_h}{\sup} \frac{b(w-\Pi_h w,\mu)}{\|\mu\|_{\M}}
\leq \frac{\|b\|}{\beta_h}\|w-\Pi_h w\|_{\X}.
\end{align*}
Thus, thanks to the triangle inequality, it follows that
\begin{equation}\nonumber
\norm{w-\tilde{\mathcal{P}}_hw-\Pi_h w}_{X}\leq \left(1+\frac{\norm{b}}{\beta_h}\right)\norm{w-\Pi_h w}_{X}.
\end{equation}
Hence, if we define
\[
\mathcal{P}_h w:=\tilde{\mathcal{P}}_h w+\Pi_h w \quad \forall w\in X,
\]
then
\[ 
\norm{w-\mathcal{P}_hw}_{X}\leq \left(1+\frac{\norm{b}}{\beta_h}\right)\norm{w-\Pi_h w}_{X},v
\]
and using \eqref{infimoproyector}, we deduce
\begin{equation}\label{infimoPh}
\norm{w-\mathcal{P}_hw}_{X}\leq C\left(1+\frac{\norm{b}}{\beta_h}\right)\inf_{z\in \X_h}\|w-z\|_{\X}\qquad\forall w\in X.
\end{equation}
Therefore, we can consider the following split of the error
\begin{equation}\label{splitE}
e_h^n:=u(t_n)-u_h^n=\rho_h^n+\sigma_h^n,\qquad n=1,\dots,N,
\end{equation}
where
\begin{equation}
\label{rhoysigma}
\rho_h(t):=u(t)-\mathcal{P}_h u(t),\quad 
\rho_h^n:=\rho_h(t_n),\quad\sigma_h^n:=\mathcal{P}_h u(t_n)-u_h^n.
\end{equation}

Next, in order to obtain the convergence of the method, we first prove the estimate for the last term in \eqref{splitE} 
as in the following Lemma.
\begin{lemma}\label{lemmasigma}
For $n=1,\dots,N$, let $\rho_h^n$ and $\sigma_h^n$ as in \eqref{rhoysigma} and 
\[\displaystyle\tau^n:=\frac{u(t_n)-u(t_{n-1})}{\Delta t}-\partial_t u(t_n).\] 

Assume that 
{$\lambda\in\mathcal{C}^1(0,T;\M)$} and $\{\xi_h\}_{h>0}$, $\{\alpha_h\}_{h>0}$ (see H7) are bounded uniformly in $h$, then provided $\Delta t$ is small enough, there exists a constant $C>0$ independent of $h$ and $\Delta t$, such that 
\begin{equation}\label{result_lemma_sigma}
\begin{split}
&\langle R\sigma_h^n,\sigma_h^n\rangle_{\Y}
+\Delta t\,\sum_{k=1}^n\|\sigma_h^k\|^2_{\X}\\
&\qquad\leq C\left(\langle R\sigma_h^0,\sigma_h^0\rangle_{\Y}
+\Delta t\,\sum_{k=1}^N \left[\|\tau^k\|^2_{\Y}+\|\dif\rho_h^k\|^2_{\Y}+\|\rho_h^k\|^2_{\X}
+\left(\underset{v\in V_h}{\sup}\frac{b(v,\partial_t\, \lambda(t_k))}{\|v\|_{\X}}\right)^2\right]\right).
\end{split}
\end{equation}
\end{lemma}	 
\begin{proof}
Let $n\in\{1,\dots,N\}$ and $k\in\{1,\dots,n\}$. Using \eqref{PC1} and \eqref{PD1}, it is straightforward to show that
\begin{equation*}
\langle R\dif \sigma_h^k,v\rangle_{\Y}+\langle A\sigma_h^k,v\rangle_{\X}
=\langle R\tau^k,v\rangle_{\Y}-\langle R\dif\rho_h^k,v\rangle_{\Y}
-\langle A\rho_h^k,v\rangle_{\X}-b(v,\partial_t\, \lambda(t_k)) \qquad \forall v\in V_h.
\end{equation*}
Then, by taking $v:=\sigma_h^k\in V_h$ in this last identity, we have 
\begin{equation}\label{aux1}
\langle R\dif \sigma_h^k,\sigma_h^k\rangle_{\Y}
+\langle A\sigma_h^k,\sigma_h^k\rangle_{\X}
=\langle R\tau^k,\sigma_h^k\rangle_{\Y}
-\langle R\dif\rho_h^k,\sigma_h^k\rangle_{\Y}
-\langle A\rho_h^k,\sigma_h^k\rangle_{\X}
{-b(\sigma_h^k,\partial_t\, \lambda(t_k))}.
\end{equation}
Let $\xi:=\displaystyle\sup_{h>0}\xi_h>0$ and $\alpha:=\displaystyle\inf_{h>0}\alpha_h>0$. Since $R$ is monotone (see H2), it is easily seen that 
\begin{align*}\label{tele-a}
\langle R\dif \sigma_h^k,\sigma_h^k\rangle_{\Y}
&\geq \frac{1}{2\Delta t}\left[\langle R\sigma_h^k,\sigma_h^k\rangle_{\Y}
-\langle R\sigma_h^{k-1},\sigma_h^{k-1}\rangle_{\Y}
\right]\ ,
\end{align*}
thus, from \eqref{aux1} we deduce
\begin{multline*}
\frac{1}{2\Delta t}\left[\langle R\sigma_h^k,\sigma_h^k\rangle_{\Y}
-\langle R\sigma_h^{k-1},\sigma_h^{k-1}\rangle_{\Y}
\right]
+\alpha \|\sigma_h^k\|^2_{\X}-\xi\,\langle R\sigma_h^k,\sigma_h^k\rangle_{\Y}\\
\leq \langle R\tau^k,\sigma_h^k\rangle_{\Y}
-\langle R\dif\rho_h^k,\sigma_h^k\rangle_{\Y}
-\langle A\rho_h^k,\sigma_h^k\rangle_{\X}
{-b(\sigma_h^k,\partial_t\, \lambda(t_k))}.
\end{multline*}
Now, since the operator $R$ is monotone, it satisfies the following Cauchy-Schwarz type inequality
\begin{equation*}\label{cauchy-R}
|\langle Rv,w\rangle_{\Y}|^2\leq \langle Rv,v\rangle_{\Y}  \langle Rw,w\rangle_{\Y}, 
\end{equation*}
then, we have
\begin{align*}
\langle R\sigma_h^k,\sigma_h^k\rangle_{\Y}
-\langle R\sigma_h^{k-1},\sigma_h^{k-1}\rangle_{\Y}
&+\alpha\Delta t \|\sigma_h^k\|^2_{\X}\\
&\leq 
(1+2\xi)\Delta t \langle R\sigma_h^k,\sigma_h^k\rangle_{\Y}
+ 2\Delta t\langle R\tau^k,\tau^k\rangle_{\Y} 
+ 2\Delta t\langle R\dif\rho_h^k,\dif\rho_h^k\rangle_{\Y}
\\
&\quad\quad+\frac2\alpha \Delta t\|A\|^2\|\rho_h^k\|^2_X 
+\frac2\alpha \Delta t
\left(\underset{v\in V_h}{\sup}\frac{b(v,\partial_t\, \lambda(t_k))}{\|v\|_{\X}}\right)^2,
\end{align*}
and using the continuity of $R$, it follows that
\begin{multline*}
\langle R\sigma_h^k,\sigma_h^k\rangle_{\Y}
-\langle R\sigma_h^{k-1},\sigma_h^{k-1}\rangle_{\Y}
+\alpha\Delta t \|\sigma_h^k\|^2_{\X}\\
\leq 
(1+2\xi)\Delta t \langle R\sigma_h^k,\sigma_h^k\rangle_{\Y}
+C\,\Delta t
\left(
\|\tau^k\|_Y^2 
+ \|\dif\rho_h^k\|_Y^2
+\|\rho_h^k\|^2_X 
+\left(\underset{v\in V_h}{\sup}\frac{b(v,\partial_t\, \lambda(t_k))}{\|v\|_{\X}}\right)^2
\right).
\end{multline*}
Then, by summing over $k$, we obtain
\[
\langle R\sigma_h^n,\sigma_h^n\rangle_{\Y}-\langle R\sigma_h^0,\sigma_h^0\rangle_{\Y}
+\alpha\Delta t\sum_{k=1}^n\|\sigma_h^k\|^2_{\X}
\leq
(1+2\xi)\,\Delta t\sum_{k=1}^n\langle R\sigma_h^k,\sigma_h^k\rangle_{\Y}
+C\,\Delta t\,\sum_{k=1}^n\Theta_k^2,
\]
where
\[
\Theta_k^2:=\|\tau^k\|^2_{\Y}+\|\dif\rho_h^k\|^2_{\Y}+\|\rho_h^k\|^2_{\X}
+\left(\underset{v\in V_h}{\sup}\frac{b(v,\partial_t\, \lambda(t_k))}{\|v\|_{\X}}\right)^2\ .
\]
Hence, if $1-(1+2\xi)\,\Delta t \geq \frac12$ then
\begin{equation}\label{cuatrounoseis}
\langle R\sigma_h^n,\sigma_h^n\rangle_{\Y}
+2\alpha\Delta t\sum_{k=1}^n\|\sigma_h^k\|^2_{\X}
\leq 2\langle R\sigma_h^0,\sigma_h^0\rangle_{\Y}
+ 2(1+2\xi)\,\Delta t\sum_{k=1}^{n-1}\langle R\sigma_h^k,\sigma_h^k\rangle_{\Y}
+ C\,\Delta t\,\sum_{k=1}^n\Theta_k^2,
\end{equation}
and, in particular
\[
\langle R\sigma_h^n,\sigma_h^n\rangle_{\Y}
\leq
2\langle R\sigma_h^0,\sigma_h^0\rangle_{\Y}
+
2(1+2\xi)\,\Delta t\sum_{k=1}^{n-1}\langle R\sigma_h^k,\sigma_h^k\rangle_{\Y}
+C\, \Delta t\,\sum_{k=1}^n\Theta_k^2.
\]
Consequently, by using the discrete Gronwall's inequality, we obtain
\[
\langle R\sigma_h^n,\sigma_h^n\rangle_{\Y}
\leq
C\,\left(
\langle R\sigma_h^0,\sigma_h^0\rangle_{\Y}
+\Delta t\,\sum_{k=1}^n\Theta_k^2
\right),
\]
for $n=1,\ldots,N$. Thus,
\begin{align*}
\Delta t\sum_{k=1}^{n-1}\langle R\sigma_h^k,\sigma_h^k\rangle_{\Y}
&\leq
C\,\Delta t\sum_{k=1}^{n-1}\left(
\langle R\sigma_h^0,\sigma_h^0\rangle_{\Y}
+\Delta t\,\sum_{j=1}^k\Theta_j^2
\right)
\leq C\,
\left(
\langle R\sigma_h^0,\sigma_h^0\rangle_{\Y}
+\Delta t\sum_{j=1}^{n-1}\Theta_j^2
\right),
\end{align*}
and finally, by substituting this inequality into \eqref{cuatrounoseis}, it follows \eqref{result_lemma_sigma}.
\end{proof}

\begin{theorem}\label{erroru}
Under the assumptions of Lemma~\ref{lemmasigma}, if $\{\beta_h\}_{h>0}$ is bounded uniformly in $h$ and $u\in \H^1(0,T;\X)\cap\H^2(0,T;\Y)$ then there exists a constant $C>0$ independent of $h$ and $\Delta t$, such that
\begin{align*}
\max_{1 \leq n \leq N}&\langle R(u(t_n)-u_h^n),u(t_n)-u_h^n\rangle_{\Y}
+\Delta t\sum_{n=1}^N\|u(t_n)-u_h^n\|^2_{\X}\\
&\leq C \left\{
\|u_0 -u_{0,h} \|_{\Y}^2 
+\max_{0 \leq n \leq N}\left( \inf_{z\in \X_h}{\|u(t_n)-z\|_{\X}}\right)^2
\right. \\ 
&\quad\quad\left.
+\int_0^T  \left(\inf_{z\in\X_h} \|\partial_{t} u(t)-z\|_{\X}\right)^2 dt 
+ (\Delta t)^2 \| \partial_{tt} u\|^2_{\L^2(0,T;\Y)} 
+ \Delta t \sum_{n=1}^N \left(\inf_{\mu\in\M_h}\|\partial_{t}\lambda(t_n)-\mu\|_{\M}\right)^2
\right\}.
\end{align*}
\end{theorem}
\begin{proof}
A Taylor expansion shows that 
\[
\sum_{k=1}^N \| \tau^k \|^2_{\Y}
=\sum_{k=1}^N \left\|\frac{1}{\Delta t} \int_{t_{k-1}}^{t_{k}} (t_{k-1}-t)\partial_{tt}u(t)\dt\right\|_{Y}^2
\leq \Delta t \int_0^T\|\partial_{tt}u(t)\|_{\Y}^2 \dt.
\]

It is easy to show that
\[
\underset{v\in V_h}{\sup}\frac{b(v,\partial_t\, \lambda(t_k))}{\|v\|_{\X}}\leq C \inf_{\mu\in\M_h} \| \partial_{t} \lambda(t_k)-\mu \|_{\M}.
\]

On the other hand, from \eqref{infimoPh}, \eqref{rhoysigma} and recalling that $\{\beta_h\}_{h>0}$ is bounded uniformly in $h$, we get
\[
\|\rho_h(t)\|_{\X}\leq C \inf_{z\in \X_h}\|u(t)-z\|_{\X}. 
\]
Furthermore, the regularity assumption about $u$ implies $\partial_{t}\mathcal{P}_h u(t)=\mathcal{P}_h(\partial_{t}u(t))$, and consequently
\[
\|\partial_{t}\rho_h(t)\|_{\X}\leq C \inf_{z\in \X_h}\|\partial_{t} u(t)-z\|_{\X}.
\]
Hence, by recalling \eqref{rhoysigma}, it is easy to check that
\begin{align*}
\Delta t\sum_{k=1}^N \|\dif \rho_h^k \|^2_{\Y} 
=\sum_{k=1}^N\left\|\frac{1}{\Delta t}\int_{t_{k-1}}^{t_k}\partial_t\rho_h(t)\dt\right\|_Y^2
\leq\sum_{k=1}^N \int_{t_{k-1}}^{t_k}\|\partial_t\rho_h(t)\|_Y^2\dt
\leq C \int_0^T\|\partial_t\rho_h(t)\|_{\X}^2 dt
\end{align*}

Finally, by writing $\sigma_h^0=e_h^0-\rho_h^0$ and using the fact that $R$ is self-adjoint and  monotone from
from third equation the Problem \ref{PC1}, it follows that
\[
\langle R\sigma_h^0,\sigma_h^0\rangle_{\Y}\leq 2\langle R(u_0-u_{0,h}),u_0-u_{0,h}\rangle_{\Y}+2\langle R\rho_h^0,\rho_h^0\rangle_{\Y}.
\]

Combining these inequalities and Lemma~\ref{lemmasigma}, the result follows from the fact that $u(t_n)-u_h^n=\rho_h^n+\sigma_h^n$ (see \eqref{splitE}) and the triangle inequality. 
\end{proof}

\subsection{Error estimates for Lagrange multiplier $\lambda$}
The following theorem is the analogous result to Theorem~\ref{erroru}, but
for the case of the Lagrange multiplier $\lambda$. 
\begin{theorem}\label{error-lambda}
Under the assumptions of Theorem~\ref{erroru}, if $f\in\H^1(0,T;\X')$ then there exist a constant $C>0$ independent of $h$ and $\Delta t$, such that 
\begin{multline*}
\Delta t\sum_{n=1}^N \|\lambda(t_n)-\lambda_h^n\|_{\M}^2
\leq C
\left((\Delta t)^2 \|\partial_t f\|^2_{\L^2(0,T;\X')}
+(\Delta t)^2 \|\partial_t u\|^2_{\L^2(0,T;\X)}
\right.
\\
\left.
+ \Delta t\sum_{n=1}^N\left(\underset{\mu\in\M_h}{\inf}\|\lambda(t_n)-\mu\|_{\M}\right)^2
+ \Delta t\sum_{n=1}^N \|u(t_n)-u_h^n\|^2_{\X}\right)
\end{multline*}
\end{theorem}

\begin{proof}
Let $n\in\{1,\dots,N\}$. By integrating  first equation the Problem \ref{PC1} on $[0,t_n]$, we have
\begin{align*}
\int_0^{t_n}\frac{d}{dt}\left[\langle Ru(t),v\rangle_{\Y}
+ b(v,\lambda(t))\right]\dt+\int_0^{t_n}a(u(t),v)\dt
=\int_0^{t_n}\langle f(t),v\rangle_{\X}\dt 
&& \forall v\in\X,
\end{align*}
then,  using the initial condition for $u$ and  recalling that $\lambda(0)=0$ (see Theorem \ref{princonti}),  it follows that
$\lambda(t_n)$ satisfies
\begin{equation}\label{probltn}
b(v,\lambda(t_n))=\langle L(t_n),v\rangle_{\X}\qquad \forall v\in\X, 
\end{equation}
where
\[
\langle L(t_n),v\rangle_{\X}
:=\left\langle \int_0^{t_n}f(t)\dt,v\right\rangle_{\X}
-\langle R(u(t_n)-u_0),v\rangle_{\Y}
-\left\langle A\int_0^{t_n}u(t)\dt,v\right\rangle_{\X}.
\]

On the other hand, by summing  \eqref{PD1} from $1$ to $n$ and using the Problem \ref{PD1}, we deduce that 
$\lambda_h^n$ verifies
\begin{equation}\label{problhn}
b(v,\lambda_h^n)=\langle L_h^n,v\rangle_{\X}\qquad \forall v\in\X_h,
\end{equation}
with
\[
\langle L_h^n,v\rangle_{\X}
:=\left\langle \Delta t \sum_{k=1}^n f(t_k),v\right\rangle_{\X}
-\langle R(u_h^n-u_{0,h}),v\rangle_{\Y}
-\left\langle A\left(\Delta t \sum_{k=1}^n u_h^k\right),v\right\rangle_{\X}.
\]

Now, we need to consider the problem
\begin{problem}\label{probtlhn}
 Find $\widetilde{\lambda}_h^n\in\M_h$ such that 
\begin{equation*}
b(v,\widetilde{\lambda}_h^n)=\langle L(t_n),v\rangle_{\X}\qquad \forall v\in\X.
\end{equation*}
\end{problem}
We can check that $L(t_n)\in {V_h^\perp:=\{v\in\X\,:\,b(v,\mu)=0\ \forall \mu\in\M_h \}}$, i.e.,
\[
\langle L(t_n),v\rangle_{\X}=0\qquad\forall v\in V_h,
\]
thus, the Problem~\ref{probtlhn} is a well-posed problem, since $b$ satisfies the discrete \textit{inf-sup} condition.

Next, we notice that by using \eqref{probltn} and \eqref{probtlhn}, the following orthogonality relationship is obtained
\[
b(v,\lambda(t_n)-\widetilde{\lambda}_h^n)=0\qquad \forall v\in\X_h,
\]
consequently, the discrete \textit{inf-sup} condition implies
\[
\|\widetilde{\lambda}_h^n-\mu\|_{\M}
\leq\frac{1}{\beta_h}\ \underset{v\in\X_h}{\sup}\frac{b(v,\widetilde{\lambda}_h^n-\mu)}{\|v\|_{\X}}
=\frac{1}{\beta_h}\ \underset{v\in\X_h}{\sup}\frac{b(v,\lambda(t_n)-\mu)}{\|v\|_{\X}}
\leq\frac{\|b\|}{\beta_h}\|\lambda(t_n)-\mu\|_{\M},
\]
and therefore
\[
\|\lambda(t_n)-\widetilde{\lambda}_h^n\|_{\M}
\leq \left(1+\frac{\|b\|}{\beta_h}\right)\underset{\mu\in\M_h}{\inf}\|\lambda(t_n)-\mu\|_{\M}.
\]

On the other hand, by using \eqref{problhn} and \eqref{probtlhn}, it follows that
\[
\|\widetilde{\lambda}_h^n-\lambda_h^n\|_{\M}
\leq\frac{1}{\beta_h}\ \underset{v\in\X_h}{\sup}\frac{b(v,\widetilde{\lambda}_h^n-\lambda_h^n)}{\|v\|_{\X}}
=\frac{1}{\beta_h}\ \underset{v\in\X_h}{\sup}\frac{\langle L(t_n)-L_h^n,v\rangle_{\X}}{\|v\|_{\X}},
\]
hence
\[
\|\lambda(t_n)-\lambda_h^n\|_{\M}
\leq \left(1+\frac{\|b\|}{\beta_h}\right)\underset{\mu\in\M_h}{\inf}\|\lambda(t_n)-\mu\|_{\M}+\frac{1}{\beta_h}\ \underset{v\in\X_h}{\sup}\frac{\langle L(t_n)-L_h^n,v\rangle_{\X}}{\|v\|_{\X}},
\]
and finally, recalling that $\{\beta_h\}_{h>0}$ is bounded uniformly in $h$, we can conclude 
\begin{equation*}
\Delta t\sum_{n=1}^N \|\lambda(t_n)-\lambda_h^n\|_{\M}^2
\leq C\,\left( \Delta t\sum_{n=1}^N\left(\underset{\mu\in\M_h}{\inf}\|\lambda(t_n)-\mu\|_{\M}\right)^2
+\Delta t\sum_{n=1}^N\left(\underset{v\in\X_h}{\sup}\frac{\langle L(t_n)-L_h^n,v\rangle_{\X}}{\|v\|_{\X}}\right)^2\right).
\end{equation*}

It remains to estimate the last term in the previous inequality. In fact, using the definitions of 
$L(t_n)$ and $ L_h^n$, for all $v\in X_h$,  we deduce
\begin{multline*}
\left(\underset{v\in\X_h}{\sup}\frac{\langle L(t_n)-L_h^n,v\rangle_{\X}}{\|v\|_{\X}}\right)^2\\
\leq C\,\left(\|u_0-u_{0,h}\|^2_{\Y}+\left\|\int_0^{t_n}f(t)\dt-\Delta t \sum_{k=1}^n f(t_k)\right\|^2_{\X'}
+\left\|\int_0^{t_n}u(t)\dt-\Delta t \sum_{k=1}^n u_h^k\right\|^2_{\X}
+\|e_h^n\|^2_{\X}\right),
\end{multline*}
therefore,
\begin{multline}\label{estisupL}
\Delta t\sum_{n=1}^N\left(\underset{v\in\X_h}{\sup}\frac{\langle L(t_n)-L_h^n,v\rangle_{\X}}{\|v\|_{\X}}\right)^2
\leq C\,\Delta t\left(
N\|u_0-u_{0,h}\|^2_{\Y}
+\sum_{n=1}^N\left\|\int_0^{t_n}f(t)\dt-\Delta t \sum_{k=1}^n f(t_k)\right\|^2_{\X'}
\right.
\\
\left.
+\sum_{n=1}^N\left\|\int_0^{t_n}u(t)\dt-\Delta t \sum_{k=1}^n u_h^k\right\|^2_{\X}
+\sum_{n=1}^N\|e_h^n\|^2_{\X}
\right).
\end{multline}

Next, we will show that
\begin{equation}\label{estiintu}
\Delta t\sum_{n=1}^N
\left\|\int_0^{t_n}u(t)\dt-\Delta t \sum_{k=1}^n u_h^k\right\|^2_{\X}
\leq 
2T^2 (\Delta t)^2 \|\partial_t u\|^2_{\L^2(0,T;\X)}
+2 T^2 \Delta t \sum_{k=1}^N \| e_h^k \|^2_{\X}
\end{equation}
and
\begin{equation}\label{estintf}
\Delta t\sum_{n=1}^N \left\|\int_0^{t_n}f(t)\dt-\Delta t \sum_{k=1}^n f(t_k)\right\|^2_{\X'}
\leq T^2 (\Delta t)^2 \|\partial_t f\|^2_{\L^2(0,T;\X')} 
\end{equation} 
In fact, to obtain \eqref{estiintu}, first we notice that
\begin{equation*}
\begin{split}
\left\|\int_0^{t_n}u(t)\dt-\Delta t \sum_{k=1}^n u_h^k\right\|^2_{\X}
&= \left\|\left(\int_0^{t_n}u(t)\dt-\Delta t \sum_{k=1}^n u(t_k)\right)+\Delta t \sum_{k=1}^n e_h^k\right\|^2_{\X}\\
&\leq { 2\left\|\int_0^{t_n}u(t)\dt-\Delta t \sum_{k=1}^n u(t_k)\right\|^2_{\X}
+2T\Delta t \sum_{k=1}^n \|e_h^k\|^2_{\X}}
\end{split}
\end{equation*}
 and
\begin{align*}
\left\|\int_0^{t_n}u(t)\dt-\Delta t \sum_{k=1}^n u(t_k)\right\|^2_{\X}
&=\left\|\sum_{k=1}^n\int_{t_{k-1}}^{t_k}(u(t)-u(t_k))\dt\right\|^2_{\X}
\leq \left(\sum_{k=1}^n\int_{t_{k-1}}^{t_k}\|u(t)-u(t_k)\|_{\X}\dt\right)^2
\\
&
\leq 
n\sum_{k=1}^n\left(\int_{t_{k-1}}^{t_k}\|u(t)-u(t_k)\|_{\X}\dt\right)^2
\leq n \Delta t\sum_{k=1}^n\int_{t_{k-1}}^{t_k}\|u(t)-u(t_k)\|^2_{\X}\dt
\\
&
= T \sum_{k=1}^n\int_{t_{k-1}}^{t_k}\left\|-\int_t^{t_k} \partial_t u(s)\ds\right\|^2_{\X}\dt
\leq T (\Delta t)^2\int_0^T\|\partial_t u(s)\|^2_{\X}ds,
\end{align*}
which implies \eqref{estiintu}. 
Next, we can apply similar computations to deduce \eqref{estintf}. 
Finally, combining \eqref{estisupL}--\eqref{estintf} and the fact that $e_h^n=u(t_n)-u_h^n$, we conclude the proof.
\end{proof}
\section{Applications}\label{aplicaciones}

\subsection{The time-dependent Stokes problem}
{The time-dependet Stokes is a fundamental model of viscous flow.
This problem arises from neglecting the nonlinear terms in \textit{Navier-Stokes} (see \cite{BR}). Stokes flows are important in lubrication theory, in porous media flow, biology applications (see \cite{panton})}. Now, we proceed to study 
the numerical approximation the time-dependent Stokes problem.
In this way, we consider $\Omega\subset \R^d$ be an open, bounded and connected, where $d$ either $2$ or $3$ the space dimension. The boundary of $\Omega$ is denoted by $\Gamma:=\partial\Omega$ and assumed to be Lipschitz continuous. 
In \cite{AGL} we have studied the well posedness of the following  variational formulation for time-dependent Stokes problem given by  

\begin{problem} \label{variationalstokes}
 Find $\un\in\ldoshounod$ and $P\in\H^1(0,T;\lodoso)$
such that
\begin{align*}
\frac{d}{dt}\left( \int_{\Omega} \un(t)\cdot \vn-\int_\Omega P(t)\dive \vn \right)+ \nu \int_{\Omega }\nabla\un(t) : \nabla\vn
=\int_{\Omega}\boldsymbol{f}(t)\cdot\vn&\qquad \forall\vn\in\hounoO^d,\\
\int_\Omega q\dive\un = 0&\qquad\forall q\in\lodoso,\\
\un(\cdot,0)=\un_0(\cdot) &\qquad \mathrm{in}\ \Omega,
\end{align*}
\end{problem}
where tensor product $\boldsymbol{z}:\wn$ is given by $\displaystyle
\boldsymbol{z}:\wn:=\sum_{i=1}^d\sum_{j=1}^d \boldsymbol{z}_{ij}\wn_{ij};\,\,\text{for any}\,\,\boldsymbol{z},\wn\in \L^2(\Omega)^{d\times d} $.

In order to obtain the fully-discrete approximation of Problem~\ref{variationalstokes}, so 
we want to use finite element subspaces to define $X_h$ and $M_h$, the corresponding families of finite 
dimensional subspaces of $X:=\houno^d$ and $M:=\lodoso$, respectively. 
To this aim, in what follows we assume that $\Omega$ is a Lipschitz polygon if $d=2$ or a Lipschitz polyhedra if $d=3$. 
Likewise, let $\set{\mathcal{T}_h}_h$ be a regular family of triangles meshes of $\Omega$ if $d=2$ or of  
tetrahedral meshes of $\Omega$ for the case $d=3$. 

The spaces $X_h$ and $M_h$ should be respectively finite element subspaces of $\houno^d$ and $\lodoso$ satisfying the discrete
inf-sup condition required for the assumption {H7} (see Section~\ref{discreto}), \textit{i.e.},
\begin{equation}\label{d-inf-sup-S}
\sup_{\vn\in\X_h}\dfrac{-\int_\Omega {q\dive\vn}}{\norm{\vn}_{\houno^d}}\geq \beta\norm{q}_{\ldoso}\qquad\forall q\in\M_h.
\end{equation}

For the sake of simplicity, 
we only consider the pair of finite element subspaces so-called 
\textit{the MINI finite element}, which was introduced by Arnold, Brezzi and Fortin \cite{mini}.
In the MINI element the discrete space for the velocity and pressure are respectively defined by
\begin{equation}\label{Xh-mini}
\begin{split}
\X_h
:=\set{\vn\in\left[{C}(\Omega)\cap\houno\right]^d:\: \vn\vert_{K}\in\left[\poli_1\oplus\mathbb{B}_{d+1}\right]^d\quad\forall K\in\mathcal{T}_h}
\end{split}
\end{equation}
and
\begin{equation}\label{Mh-mini}
\begin{split}
\M_h:=\set{\vn\in{C}(\Omega)\cap\lodoso:\: \vn\vert_{K}\in\poli_1\quad\forall K\in\mathcal{T}_h},
\end{split}
\end{equation}
where $\mathbb{B}_{d+1}$ are the bubbles functions, which are defined by
\begin{equation}\nonumber
\mathbb{B}_{d+1}
:=
\set{\gamma\varphi(x_1,\ldots,x_d):\: \gamma\in\R,\quad \varphi(x_1,\ldots,x_d):=x_1\ldots x_d(1-x_1-\ldots-x_d)}
.
\end{equation}

The bubble part of the of the velocity is needed to stabilized the formulation, 
namely to satisfy the discrete inf-sup condition, which is shown in the following lemma.

\begin{lemma}\label{mini-stable}
Let $\X_h$ and $\M_h$ given by \eqref{Xh-mini} and \eqref{Mh-mini} respectively. 
Then, there exists $\beta>0$ independent of $h$ satisfying the discrete inf-sup condition \eqref{d-inf-sup-S}.
\end{lemma}
\begin{proof}
See, for instance, \cite[Lemma~4.20]{guermond}.
\end{proof}
By using the discrete subspaces given by the MINI elements, we can consider the following fully-discrete 
approximation of Problem \ref{variationalstokes} given $\un_{0,h}\in\X_h$ and by denoting
\begin{equation}\nonumber
\un_h^0:=\un_{0,h},\qquad P^0_h:=0,
\end{equation}
\begin{problem}\label{discretestokes}
Find $\unhn\in\X_h$ and $\Phn\in\M_h$ for $n=1,\ldots,N$,  such that
\begin{align*}
\int_{\Omega} \left(\frac{\unhn-\unhnm}{\Delta t}\right)\cdot \vn 
 -\int_\Omega \left(\frac{\Phn-\Phnm}{\Delta t}\right)\dive \vn
 + \nu \int_{\Omega }\nabla\unhn : \nabla\vn
=\int_{\Omega}\boldsymbol{f}(t_n)\cdot\vn
&\qquad \forall\vn\in\X_h,
\\
\int_\Omega q\dive\unhn = 0&\qquad\forall q\in\M_h.
\end{align*}
\end{problem}
Since the MINI element satisfies the discrete inf-sup condition (see Lemma~\ref{mini-stable}), 
the property $\mathrm{H7}$ 
from Section~\ref{discreto} holds true. Hence, to deduce that the Problem~\ref{discretestokes}
has a unique solution, it only remains to prove $\mathrm{H6}$, \textit{i.e.}, the Garding-like inequality
\begin{equation}\label{garding-d-ST}
\nu\int_{\Omega}\abs{\nabla\vn}^2 
+ \gamma_h\int_{\Omega}\abs{\vn}^2 
\geq \alpha_h\norm{\vn}_{\hounoO^3}^2\qquad\forall\vn\in V_h,
\end{equation}
where $V_h$ is the discrete kernel
\begin{equation}\nonumber
V_h:=\set{\vn\in\X_h:\: \int_{\Omega}q\dive\vn=0\quad\forall q\in\M_h}.
\end{equation}
Moreover, \eqref{garding-d-ST} holds uniformly on $h$ (more exactly with $\gamma_h=0$ and $\alpha=\nu$),
since the norm in $\hounoO^3$ is precisely
\[
\norm{\vn}_{\hounoO^3}=\left(\int_{\Omega}\abs{\nabla\vn}^2\right)^{1/2}.
\]

Consequently, Problem \ref{discretestokes} has a unique solution $(\unhn)_{n=1}^N\subset\X_h$
and $(\Phn)_{n=1}^N\subset\M_h$.

Our next goal is to obtain \textit{C\'ea}-like error estimates for the fully-discrete approximation 
of the Stokes problem. 
We will first obtain error estimates for the approximation of the velocity. 
In fact, we have the following result for a direct application of Theorem~\ref{erroru}.
\begin{theorem}\label{cea-u-ST}
Suppose the assumptions of Theorem 2.1 studied in  \cite{AGL} and let $(\un,P)$ and  $(\unhn,\Phn)_{n=1}^N$
be the solutions of Problem \ref{variationalstokes} and Problem \ref{discretestokes},
respectively.
Furthermore, assume that 
\begin{equation*}\label{regu-uP}
\begin{split}
{\un\in \H^1(0,T;\hounoO^3)\cap\H^2(0,T;\ldoso^3),\qquad P\in{C}^1(0,T;\lodoso).}
\end{split}
\end{equation*}
Then, {for a small enough time step $\Delta t$,} 
there exists a constant $C>0$, independent of $h$ and $\Delta t$, such that
\begin{equation}\nonumber
\begin{split}
&\max_{1\leq n\leq N}\|\un(t_n) - \unhn\|_{\ldoso^3}^2 
+ \Delta t\sum_{n=1}^{N}\|\un(t_n) - \unhn\|_{\hounoO^3}^2\\
&\quad\leq C\left\{
\norm{\un_0-\un_{0,h}}_{\ldoso^3}^2
+\max_{1\leq n\leq N}
\inf_{\vn\in\xho}\Vert\un(t_n)-\vn\Vert_{\hounoO^3}^2
\phantom{\left(\inf_{\wn\in\M_h}\|\partial_{t}\vn(t_n)-\wn\|_{\ldosod^3}\right)^2}
\right.\\
&\qquad\left.\phantom{\sum_{k=1}^{N}}
+\Delta t
\sum_{n=1}^N\inf_{\vn\in\xho}\Vert\un(t_n)-\vn\Vert_{\hounoO^3}^2
+\int_0^T \left(
\inf_{\vn\in\xho}\Vert\partial_t\un(t)-\vn\Vert_{\hounoO^3}^2\right)\dt\right.
\\
&\qquad\left.\phantom{\sum_{k=1}^{N}}
+(\Delta t)^2\int_0^{T}\left\Vert\partial_{tt}\un(t)\right\Vert
_{\ldoso^3}^2 \dt
+ \Delta t \sum_{n=1}^N \left(\inf_{q\in\M_h}\|\partial_{t}P(t_n)-q\|_{\ldoso}\right)^2
\right\}.
\end{split}
\end{equation}
\end{theorem}

The Cea-like error estimates for the approximation of the pressure is given by the following theorem, 
which is obtained 
by a direct application of Theorem~\ref{error-lambda}.

\begin{theorem}\label{cea-P-ST}
Under the assumptions of Theorem~\ref{cea-u-ST}, if $\boldsymbol{f}\in\H^1(0,T;\ldoso^3)$ 
then, {for a small enough time step $\Delta t$,} 
there exist a constant $C>0$ independent of $h$ and $\Delta t$, such that 
\begin{equation}\nonumber
\begin{split}
&\Delta t\sum_{n=1}^N \|P(t_n)-\Phn\|_{\ldoso}^2\\
&\quad\qquad\leq C
\left[(\Delta t)^2 \|\partial_t \boldsymbol{f}\|^2_{\L^2(0,T;\ldoso^3)}
+(\Delta t)^2 \|\partial_t \un\|^2_{\L^2(0,T;\hounoO^3)}
\phantom{\left(\underset{\mu\in\M_h}{\inf}\|\lambda(t_n)-\mu\|_{\M}\right)^2}
\right.
\\
&\:\qquad\qquad\quad\left.
+ \Delta t\sum_{n=1}^N\left(\underset{q\in\M_h}{\inf}\|P(t_n)-q\|_{\ldoso}\right)^2
+ \Delta t\sum_{n=1}^N \|\un(t_n)-\unhn\|^2_{\hounoO^3}\right].
\end{split}
\end{equation}
\end{theorem}

Finally, to obtain the asymptotic error estimate for the velocity approximation, we need to consider 
the vector-valued functions of the Sobolev space $\mathrm{H}^{1+s}(\Omega)^3$ for $0<s<1$. 
We recall that the vector Lagrange interpolant $\mathcal{\boldsymbol{L}}_h \vn\in X_h$ is obtained
by taking the scalar Lagrange interpolant  of each component, thus it is
well defined for all $\vn\in\mathrm{H}^{1+s}(\Omega)^3\cap\houno^3$. 
Moreover, \[\mathcal{\boldsymbol{L}}_h:\mathrm{H}^{1+s}(\Omega)^3\cap\houno^3\to\X_{1,h},\]
where $\X_{1,h}$ is the proper subspace of $\X_h$ given by
\begin{equation}\nonumber
\begin{split}
\X_{1,h}
:=\set{\vn\in\left[{C}(\Omega)\cap\houno\right]^d:\: \vn\vert_{K}\in\poli_1^d\quad\forall K\in\mathcal{T}_h}.
\end{split}
\end{equation}
Furthermore, the following estimate holds true {(see, for instance, \cite{ciarlet})}
\begin{equation}\nonumber
\norm{\vn-\mathcal{\boldsymbol{L}}_h \vn}_{\houno^3} \leq C
h^{s} \norm{\vn}_{\mathrm{H}^{1+s}(\Omega)^3}
\qquad \forall \vn\in\mathrm{H}^{1+s}(\Omega)^3\cap\houno^3. 
\end{equation}

On the other hand, to obtain the asymptotic error estimate for the pressure approximation, 
we notice that the scalar Lagrange interpolant 
$\mathcal{{L}}_h:\mathrm{H}^{1+s}(\Omega)\cap\lodoso\to\M_h$
verifies
\begin{equation}\nonumber
\norm{q-\mathcal{{L}}_hq}_{\ldoso} \leq C
h^{s} \norm{q}_{\mathrm{H}^{1+s}(\Omega)}
\qquad \forall q\in\mathrm{H}^{1+s}(\Omega)\cap\lodoso. 
\end{equation}


Consequently, we have the following result which shows the asymptotic convergence 
of the fully-discrete approximation for the Stokes problem.

\begin{corollary}
Let $0<s<1$ a fixed index. Under the assumptions of Theorem \ref{cea-P-ST}. 
Then, if we define
\[
\un_{0,h}:=\mathcal{\boldsymbol{L}}_h\un_0,
\]
{for a small enough time step $\Delta t$,} 
there exist a constant $C>0$ independent of $h$ and $\Delta t$, such that 
\begin{align*}
\max_{1\leq n\leq N}\|\un(t_n) - \un_h^n\|_{\ldoso^3}^2 
&+ \Delta t\sum_{n=1}^{N}\|\un(t_n) - \un_h^n\|_{\hounoO^3}^2\\
&
+ \Delta t\sum_{n=1}^N \|P(t_n)-\Phn\|_{\ldosod}^2
\:\leq \:C
\left[
h^{2s} + (\Delta t)^2
\right].
\end{align*}
\end{corollary}
\begin{proof}
Corollary follows immediately by combining  Theorem~\ref{cea-u-ST} and Theorem~\ref{cea-P-ST}.
\end{proof}

\subsection{Eddy current problem}
The eddy current problem arises when the displacement currents can
be dropped from Maxwell's equation (see \cite[Chapter 9]{bossavit} and \cite{AB}). The aim of the eddy current problem is to determine the eddy currents induced in a three-dimensional conducting domain represented by an open and bounded set $\oc$. In this way,
let $\Omega\subset \R^3$ be a bounded domain simply connected with a connected boundary $\partial\Omega$. We suppose
that $\Omega$ is divided two regions: a conductor domain $\oc$ and insulator domain $\od$, where $\od=\Omega\setminus\oc$.

{
We refer some papers about the numerical analysis for the time-dependent eddy current model in 
bounded domains containing conductor and isolator materials \cite{BLRS4,WB,ZCWNM06}. 
These articles deal with the case where 
the conducting materials are strictly contained in a three-dimensional domain. In general, these formulations present natural or/and 
essential boundaries conditions depending on the main variable. For a treatment of time-dependent eddy current model with current or voltage excitations, 
we refer the reader to \cite{BLRS1,blrs:13}.  Other parabolic formulations for the time-dependent eddy current problem posed in the conductor domain 
can be founded in \cite{CR,A4}.
}

The abstract theory developed in Section~\ref{discreto} and  in Section~\ref{estimates} can be used to study the fully discrete  mixed formulations proposed for the eddy current model in \cite{MS1,A1,A2,blrs:13}, but we will only focus in the numerical analysis for the formulations studied in \cite{A1} and \cite{blrs:13}. 
The well-posedness of the continuous variational formulations corresponding to \cite{A1} and \cite{blrs:13} was proved in \cite[Theorem 3.2 and Theorem 3.4]{AGL}, respectively.

In both cases, the formulation of  the eddy current problem was deduced in  terms of a time-primitive of the electric field, i.e., in terms of the unknown $\un$ given by
\begin{equation*}\label{relauE}
\un(\xn,t):=\int_0^t \bE(\xn,s)ds.
\end{equation*}
In \cite{A1}, it was proposed for an internal conductor $\partial\oc\cap \partial\Omega = \emptyset$, with the boundary condition
\begin{align*}
\bE\times \nn=\cero \quad \text{on}\quad  [0,T]\times \partial\Omega.
\end{align*}
Furthermore, to determine the uniqueness of $\bE$  they must add the following conditions:
\begin{align*}
\dive(\varepsilon\bE)=0&\qquad\textrm{in }\od\times[0,T),
\\
\int_{\Sigma_i}\varepsilon\bE|_{\od}\cdot\bn \:=\:0&
\qquad \textrm{in }[0,T),\quad i=1,\dots,M_I,
\end{align*}
where $\Sigma_i$, $i=1,...,M_I,$ are the connected components of $\Sigma:=\partial\oc$.

On the other hand, in \cite{blrs:13} the eddy current model was analyzed with input currents intensities as source data. 
{
In this case, the conductor can be not strictly contained in $\Omega$ and 
$\Gamma_C:=\partial \Omega_C\cap \partial \Omega$ can be not empty.
They have assumed $\partial \Omega_C\cap \partial \Omega=\GE\cup\GJ$, 
where $\GJ$ is the boundary associated with the current entrance (the conductor is connected to 
a source current) and $\GE$ is is the boundary associated with the current exit.  Furthermore, $\GJ=\bigcup_{k=1}^{k=N}{\GJ^k}$, 
where $\GJ^k$ are the connected components of $\GJ$ (see Figure~\ref{sketch01}).
 }

\begin{figure}[!ht]
	\begin{center}
		\includegraphics*[width=7cm]{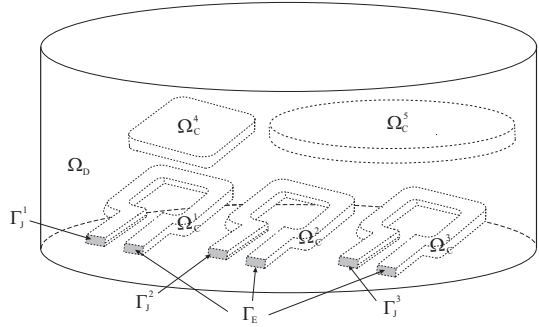}	
	\end{center}
	\caption{Sketch of the domain}
	\label{sketch01}
\end{figure}
The intensities of the input current are imposed by
\[
\int_{\GJ^n}\sigma\bE\cdot\bn=I_n \quad\text{in }[0,T],\quad n=1,\ldots,N,
\]
where $I_n$ is the current intensity through the surface $\GJ^n$, and the following boundary conditions was proposed
\begin{align*}
\bE\times\nn=\cero &\quad\text{on }[0,T]\times\GE,
\\
\bE\times\nn=\cero &\quad\text{on }[0,T]\times\GJ,
\\
\mu\bH\cdot\nn=0 &\quad\text{on }[0,T]\times\partial\Omega,
\end{align*}
where $\bH$ is the magnetic field. In this case, $\bE$ is uniquely determined provided that
\begin{align*}
\dive(\varepsilon \bE)=0&\quad \text{in } [0,T]\times\od,
\\
\epsilon \bE|_{\od}\cdot\bn =g &\quad \text{on }[0,T]\times\Gd,
\\
\int_{\GI^k}\varepsilon\bE|_{\od}\cdot\bn=0, &\quad
k=2,\dots,M_{I}, \quad \mbox{in }[0,T],
\end{align*}
{where $\Gd:=\partial\Omega\cap\partial\od$ and $\GI^k:=\partial\oc^k\cap\partial\od$, 
$k=1,\dots,M_{I},$ are the connected components of 
the interface boundary $\GI$, which is the boundary between the conductor domain and the insulator domain 
(see Figure~\ref{domain}, $M_{I}=5$). Furthermore, $g$ is an additional data.}

To show the application of the  abstract theory in the models studied in \cite{A1} and \cite{blrs:13} is necessary to introduce the 
following functional spaces:
\begin{align*}
\hcurlo&:=\left\{\vn\in\ldoso^3:\: \curl\vn\in\ldoso^3\right\},\\
\hocurlo&:=\left\{\vn\in\hcurlo:\: \vn\times\nn=\cero \quad \text{on}\quad \partial
\Omega\right\},\\
\mathbf{H}(\curl^0;\Omega)&:=\left\{\vn\in\hcurlo:\: \curl\vn=\cero \quad \text{in}\quad
\Omega\right\}.
\end{align*}
Finally, given a subset $\Lambda\subset \partial\Omega$, we denote by 
\begin{equation*}
 \mathbf{H}_{\Lambda}(\curl;\Omega):=\left\{\vn\in\hcurlo:\: \vn\times \nn=\cero \quad \text{on}\quad
\Lambda\right\}
\end{equation*}
and 
\begin{equation*}
 \mathbf{H}_{\Lambda}(\curl^0;\Omega):=\mathbf{H}(\curl^0;\Omega)\cap \mathbf{H}_{\Lambda}(\curl;\Omega).
\end{equation*}
Similarly, for $\hdive$, $\hodive$, $\mathbf{H}_{\Lambda}(\dive;\Omega)$ and  
{
\[
\mathbf{H}_{\Lambda}(\dive^0_{\varepsilon};\Omega)=\left\{\wn\in\mathbf{H}_{\Lambda}(\dive;\Omega): \dive({\varepsilon\wn})=0 \textrm{ in }\Omega\right\}.
\]
}
From now on, we refer to the problem studied in \cite{A1} as \textit{internal conductor model} and the problem studied 
in \cite{blrs:13} as \textit{the input current model}. Furthermore, we assume that $\Omega$ and $\oc$ are Lipschitz polyhedra and that $\set{\mathcal{T}_h}_h$ is a regular family of tetrahedral meshes of $\Omega$ such that each element $K\in\mathcal{T}_h$ is contained either in $\overline{\Omega}_{\mathrm{c}}$ or in $\overline{\Omega}_{\mathrm{d}}$. As usual, $h$ stands for the largest diameter of the tetrahedra $K$ in $\mathcal{T}_h$. Finally, we suppose that the family of triangulations $\{\mathcal{T}_h(\Sigma)\}_h$ induced by $\{\mathcal{T}_h\}_h$ on $\Sigma$ is quasi-uniform.

\subsubsection{Internal conductor model}\label{IC}
Let us define 
\[
M:= \left\{\vn\in \H^1(\od): \mu|_{\partial{\Omega}}=0, \mu|_{\Sigma_i}=C_i,\, i=1,\cdots M_I \right\}
\]
endowed with usual norm in $\H^1(\od)$.
The variational formulation of eddy current model with internal conductor (see \cite{A1}) is given by

\begin{problem}\label{mix1-int}
 Find $\un\in\ldoshocurlo$ and $\lambda\in\H^1(0,T;\M)$ such that 
\begin{align*}
&\dfrac{d}{dt}\left[\int_{\oc}\sigma\un(t)\cdot\vn + \int_{\od}\varepsilon\vn\cdot\nabla\lambda(t)\right]
+ \int_{\Omega}\frac{1}{\mu}\curl\un(t)\cdot\curl\vn =\int_{\Omega}\fn(t)\cdot\vn&&\forall \vn\in\hocurlo, \\
& \int_{\od}\varepsilon\un(t)\cdot\nabla\mu = 0 &&\forall \mu\in\M ,\\
&\un(\cdot,0)=\cero\quad \textrm{in }\oc\quad \textrm{and }\quad
\lambda(0)=0 \quad \textrm{in }\od,
\end{align*}
\end{problem} 
where 
\begin{align*}
\langle \textbf{f}(t),\vn\rangle=\int_{\Omega} \textbf{f}(t)\cdot\vn=\int_{\Omega}\curl \textbf{H}_{0}\cdot\vn-\int_{\Omega} \textbf{J}(t)\cdot\vn\quad  \forall t\in [0,T]\quad \forall\vn
\in \X_h.
\end{align*}
with $\textbf{ H}_0$ the initial magnetic condition and $\textbf{J}\in\ldos(0,T;\L^2(\Omega))$.

To obtain the fully-discrete approximation for the eddy current formulation given by Problem \ref{mix1-int}, we use finite element subspaces to define $X_h$ and $M_h$, the corresponding families of finite dimensional subspaces of $X:=\hocurlo$ and $M:=\M(\od)$ respectively (see Section~\ref{discreto}). 

We define $X_h$ using N{\'e}d{\'e}lec finite elements, more precisely $X_h$ is the global N{\'e}d{\'e}lec finite elements subspace, which is defined by 
\begin{equation}\label{defXh}
X_h:=\left\{\vn\in\hocurlo:\ \vn\vert_{K}\in\nd(K)\ \forall K\in\mathcal{T}_h\right\},
\end{equation}  
where $\nd(K)$ is the local representation on $K$ of the lowest-order N{\'e}d{\'e}lec finite elements subspace
\[\nd(K):=\{\mathbf{a}\times\xn+\mathbf{b}:\mathbf{a},\mathbf{b}\in\R^{3},\ \xn\in K\}.\]

On the other hand, we use standard linear Lagrange finite elements to define $M_h$, i.e.,
\begin{equation}\label{defMh}
M_h:=\left\{\mu\in\hunod:
\ \mu\vert_{K}\in\mathbb{P}_1(K)\ \forall K\in\mathcal{T}_h,\, K\subset\adhod,
\ \mu\vert_\Gamma=0,\ \mu\vert_{\Sigma_i}
=C_i,\ i=1,\dots,I\right\},
\end{equation}
where $\mathbb{P}_m$ is the set of polynomials of degree not greater than $m$.

The corresponding fully-discrete problem of Problem \ref{mix1-int} is given by
\begin{problem}\label{mix1-dis}
 Find $\un_h^n\in\X_h$ and $\lambda_h^n\in \M_h$ for $n=1,\cdots,N$ such that 
\begin{align*}
&\left[\int_{\oc}\sigma\frac{\un_h^n-\un_h^{n-1}}{\Delta t}\cdot\vn + \int_{\od}\varepsilon\vn\cdot\frac{\nabla\lambda_h^n-\nabla\lambda_h^{n-1}}{\Delta t}\right]
+ \int_{\Omega}\frac{1}{\mu}\curl\un_h^n\cdot\curl\vn =\int_{\Omega}\fn(t_n)\cdot\vn&&\forall \vn\in\X_h, \\
& \int_{\od}\varepsilon\un_h^n\cdot\nabla\mu = 0 &&\forall \mu\in\M_h ,\\
&\un_h^0=0\quad \textrm{in}\quad \Omega\quad \textrm{and}\quad \lambda_h^0=0 \quad \textrm{in }\quad \od.
\end{align*}
\end{problem} 

In this case, we can notice that the discrete kernel of $b$ is defined by
\begin{equation}\label{defVh}
V_h=\left\{\vn\in \xho: b(\vn,\mu)=0 \quad \forall \mu\in \moh\right\}
=\left\{\vn\in \xho: \int_{\od}\varepsilon\vn\cdot\nabla\mu=0 \quad \forall \mu\in \moh\right\}.
\end{equation}
\\
\textbf{Existence and uniqueness of fully discrete-solutions.}
\\
In order to deduce the existence and uniqueness of solution for the fully-discrete approximation, we have to show that the hypotheses  H7--H8 hold. The proof of discrete \textit{inf-sup} condition H7 is completely analogous to the deduction of its continuous version H1 inside the proof of of \cite[Theorem 3.2]{AGL}. For this reason, we only show the proof of H8. To this aim, we first need to recall the following result.

\begin{lemma}
Let $X_h(\oc):=\{\vn\vert_{\oc}:\ \vn\in X_h\}$. There exists a bounded and linear mapping $\E_h:X_h(\oc)\to V_h$ satisfying:
\begin{enumerate}
\item[a)] $\E_h$ is bounded uniformly in $h$. 
\item[b)] $(\E_h\vn_c)\vert_{\oc}=\vn_c$ for all $\vn_c\in X_h(\oc)$.
\end{enumerate} 
\end{lemma}
\begin{proof}
See \cite[Lemma~5.3]{A1}.
\end{proof}

\begin{lemma}\label{lemmaGIC}
There exist positive constants $\widehat\gamma$ and $\widehat\alpha$ such that
\begin{equation}\label{dgarding}
\int_{\Omega}\dfrac{1}{\mu}\vert\curl\vn\vert^2 + \widehat\gamma\int_{\oc}\sigma\vert\vn\vert^2 
\geq \widehat\alpha\norm{\vn}_{\hcurlo}^2\qquad\forall\vn\in V_h,
\end{equation}
where $V_h$ is the discrete  kernel of $b$ given in \eqref{defVh}.
\end{lemma} 
\begin{proof}
Let $\vn\in V_h$ and define $\widetilde{\wn}:=\vn - \E_h\vn$. Then $\widetilde{\wn}\in V_h$, 
$\widetilde{\wn}=\cero$ in $\oc$ and $\widetilde{\wn}\vert_{\od}$ belongs to the subspace
\[
V_{\mathrm{d},h}=\left\{\vn\vert_{\od}: \vn\in V_h\right\}\cap\hocurlod,
\]
where $X_h(\od):=\{\vn\vert_{\od}:\ \vn\in X_h\}$. It is well-known that the semi-norm $\wn\mapsto\norm{\curl\wn}_{0,\od}$ is a norm on $V_{\mathrm{d},h}$ uniformly equivalent to the $\hcurlod$-norm (see, e.g.,~Theorem~4.7,~\cite{hiptmair}).
It follows that
\[
\begin{split}
\norm{\widetilde\wn}_{\hcurlo} &= \norm{\widetilde\wn\vert_{\od}}_{\hcurlod} 
\leq C \norm{\left(\curl\widetilde\wn\right)\vert_{\od}}_{\ldosod^3}\\
&\leq C \left\{\norm{\curl\E_h\vc}_{\ldosod^3} + \norm{\curl\vn}_{\ldosod^3}\right\}
\leq C \left\{\norm{\vc}_{\hcurloc} + \norm{\curl\vn}_{\ldosod^3}\right\}.
\end{split}
\]
Consequently,
\[
\begin{split}
\norm{\vn}_{\hcurlo}^2 &= \norm{\E_h\vn + \widetilde\wn}_{\hcurlo}^2
\leq 2\left\{\norm{\E_h\vn}_{\hcurlo}^2 + \norm{\widetilde\wn}_{\hcurlo}^2\right\}\\
&\leq C\left\{\norm{\vn}_{\hcurloc}^2 + \norm{\curl\vn}_{\ldosod^3}^2\right\},
\end{split}
\]
from which the result holds.
\end{proof}
Hence, we have deduced the existence and uniqueness of Problem \ref{mix1-dis}.
\\
\textbf{Error estimates.}
\\
Our next goal is to obtain error estimates for the fully-discrete approximation of the eddy current formulation. Since $\lambda=0$ and $\lambda_{h}^n=0$, we will only be concerned with error estimates for the main variable $\un$. In fact, we have the following result for a direct application of Theorem~\ref{erroru}. 
\begin{theorem}\label{cea-eddy}
Assume that $\un\in \H^1(0,T;\hocurlo)\cap\H^2(0,T;\ldoso^3)$. Then, there exists a constant $C>0$, independent of $h$ and $\Delta t$, such that
\begin{align*}
\max_{1\leq n\leq N}\|\un(t_n) - \un_h^n\|_{\sigma,\oc}^2 
&+ \Delta t\sum_{k=1}^{N}\|\un(t_k) - \un_h^k\|_{\hcurlo}^2\\
&\leq C\left\{\max_{1\leq n\leq N}
\inf_{\vn\in\xho}\Vert\un(t_n)-\vn\Vert_{\hcurlo}^2
+\Delta t
\sum_{n=1}^N\inf_{\vn\in\xho}\Vert\un(t_n)-\vn\Vert_{\hcurlo}^2
\right.\\
&\left.\phantom{\sum_{k=1}^{N}}
+\int_0^T \left(
\inf_{\vn\in\xho}\Vert\partial_t\un(t)-\vn\Vert_{\hcurlo}^2\right)\dt
+(\Delta t)^2\int_0^{T}\left\Vert\partial_{tt}\un(t)\right\Vert
_{\ldoso^3}^2 \dt\right\},
\end{align*}
where 
$\norm{\wn}_{\sigma,\oc}^2:=\displaystyle\int_{\oc}\sigma\vert\wn\vert^2$.
\end{theorem}
Finally, to obtain the asymptotic error estimate, we need to consider the Sobolev space
\begin{align}\label{Hr}
\mathbf{H}^r(\curl,Q):=\set{\vn\in \H^r(Q)^3:\ \curl \vn \in
\H^r(Q)^3},\quad r\geq 0
\end{align}
endowed with the norm $\norm{\vn}_{\mathbf{H}^r(\curl,Q)}^2:=\norm{\vn}_{r,Q}^2 + \norm{\curl \vn}_{r,Q}^2$, where $Q$ is either $\oc$ or $\od$. It is well known that the N\'ed\'elec interpolant $\mathcal{I}_h\vn\in X_h(Q)$ is well defined for all $\vn\in\mathbf{H}^r(\curl,Q)$ with $r>1/2$, see for instance \cite[Lemma~5.1]{alonsovalli} or \cite[Lemma~4.7]{amrouche}. We fix now an index $r>1/2$ and introduce the space
\begin{equation}\nonumber
\mathcal{X}:=\mathbf{H}^r(\curl,\Omega)\cap \hocurlo.
\end{equation}
Then, the N\'ed\'elec interpolation operator $I_h^{\mathcal{N}}:{\mathcal{X}}\to\xho$ is uniformly bounded and the following interpolation error estimate holds true; see  \cite[Lemma~5.1]{brs:02} or  \cite[Proposition~5.6]{alonsovalli}:
\begin{equation}\label{bimbo}
\norm{\vn-I_h^{\mathcal{N}}\vn}_{\hcurlo} \leq C
h^{\min\{r,1\}} \norm{\vn}_{\mathcal{X}}
\qquad \forall \vn\in\mathcal{X}. 
\end{equation}
Consequently, we have the following result which shows the asymptotic convergence of the fully-discrete approximation.
\begin{corollary}\label{coroconvucd} 
If $\un\in \H^1(0,T;\mathcal{X}\cap\hocurlo)\cap \H^2(0,T;\ldoso^3)$, there exists a constant $C>0$ independent of $h$ and $\Delta t$, such that
\begin{align*}
\max_{1\leq n\leq N}\|\un(t_n)-\un_h^n\|_{\sigma,\oc}^2 
&+ \Delta t\sum_{k=1}^{N}\|\un(t_k)-\un_h^k\|_{\hcurlo}^2
\\
&
\leq C \left\{h^{2\ell}\left(
\max_{1\leq n\leq N}
\|\un(t_n)\|_{\mathcal{X}}^2
+\|\partial_t\un\|_{\L^2(0,T;\mathcal{X})}^2\right)
+(\Delta t)^2\|\partial_{tt}\un\|_{\L^2(0,T;\ldoso^3)}^2
\right\}
\end{align*}
with $\ell:=\min\{r,1\}$.
\end{corollary}
\begin{proof}
It is a direct consequence of Theorem~\ref{cea-eddy} and the interpolation error estimate \eqref{bimbo}.
\end{proof}
\begin{remark}
By testing \eqref{aux1} with $v=\dif{\sigma_h^k}$, considering $\lambda=0$ and using similar arguments of Section \ref{estimates}, 
we obtain the following result. 
\begin{theorem}\label{errordu}
Let us assume the hyphotesis of Theorem \ref{erroru}. If the Lagrange multiplier $\lambda$ of the Problem \ref{PC1} vanishes identically
and the operator $A$ is monotone on $\X$, then there exists a constan $C>0$ independent oh $h$ and $\Delta t$ satisfying
  \begin{align*}
  &\Delta t \sum_{k=1}^n\produ{R(\partial_{t}u(t_k)-\dif u_h^k),(\partial_{t}u(t_k)-\dif u_h^k)}_{\Y}\\
 & \qquad\qquad\qquad\qquad
  \leq C\left[ \|\Pi_h u_0 -u_{0,h} \|_{\Y}^2 +\max_{1 \leq n \leq N}\left( \inf_{v\in \X}{\|u(t_{n}) - v\|_{\X}^2}\right) 
+ \Delta t\,\sum_{n=1}^N \inf_{v\in\X_h}{\|u(t_n)-v\|^2_{\X}} \right.\\ 
&\qquad\qquad\qquad\qquad\qquad\:\:\:
\left.
+\int_0^T  \left(\inf_{v\in\X} \| \partial_{t} u(t)-v \|_{\X}^2\right) dt + (\Delta t)^2 \| \partial_{tt} u \|^2_{\L^2(0,T;\Y)} 
\right].
 \end{align*}
 \end{theorem}
Finally, by applying Theorem \ref{errordu} to the eddy currents problem and using the interpolation error estimate \eqref{bimbo}, 
 we deduce the quasi-optimal convergence for the time derivative approximation.
 \begin{corollary}\label{coroconvucd2}
Let $\un$ be the solution of the Problem \ref{mix1-int}. If $\un\in \H^1(0,T;\mathcal{X}\cap\hocurlo)\cap \H^2(0,T;\ldoso^3)$, 
there exists a constant $C>0$ independent of $h$ and $\Delta t$, such that
\begin{multline*}
\Delta t \sum_{k=1}^N\|\partial_t\un(t_k)-\dif \un_h^k\|_{\ldosoc^3}\\
\leq C \left\{h^{2\ell}\left(
\max_{1\leq n\leq N}
\|\un(t_n)\|_{\mathcal{X}}^2
+\|\partial_t\un\|_{\L^2(0,T;\mathcal{X})}^2\right)
+(\Delta t)^2\|\partial_{tt}\un\|_{\L^2(0,T;\ldoso^3)}^2
\right\}
\end{multline*}
with $\ell:=\min\{r,1\}$.
\end{corollary}
 \end{remark}
{
\begin{remark}\label{variables}
At each time step $t=t_k$, we can approximate the eddy currents
$\sigma\bE(\xn,t_k)$ and the magnetic field $\bH(\xn,t_k)$ by means of  $\sigma\bE_h^k=\sigma\bar\partial\un_h^k$  and
 $\mu\bH_h^k=\curl \un_h^k-\mu\bH_0$, respectively. Then the Corollary~\ref{coroconvucd2} and the Theorem \ref{erroru}
yields the following error estimates
\[
\Delta t\sum_{k=1}^N\norm{\sigma\bE(t_k)-\sigma\bE_h^k}_{0,\oc}^2
\leq C\left[h^{2l} + (\Delta t)^2\right],
\]
\[
\Delta t\sum_{k=1}^N\norm{\mu\bH(t_k)-\mu\bH_h^k}_{0,\Omega}^2
\leq C\left[h^{2l} + (\Delta t)^2\right].
\]
\end{remark}
}
\textbf{Numerical results.}
 
Now, we will present some numerical results obtained with a MATLAB code which implements the numerical method described above. 
First, we solve a test problem with a known analytical solution. 
Next, we describe a problem with cylindrical symmetry and compare the results with those obtained with an axisymmetric code.

{Test 1: \textit{A test with known analytical solution}}

{
Let us consider $\Omega:=[0,3]^3$, $\Omega_c:=[1,2]^3$ and $T=10$. The right hand is chosen
so that 
\[
\un(x_1,x_2,x_3,t)=sin(\pi t)\begin{bmatrix}{x_1^2 x_2 x_3(x_1-3)^2(2x_2-3)(x_2-3)(x_3-3)}\\{-x_1 x_2^2 x_3(x_1-3)(2x_1-3)(x_2-3)^2(x_3-3)}\\{0}\end{bmatrix}
\]
is solution of Problem \ref{mix1-int}. Furthermore, we have assumed without loss of generality 
that
$\mu=\sigma=1$.
The numerical method has been applied with several successively refined meshes and time-steps. The computed solutions
have been compared with the analytical one, by calculating the relative percentage error in time-discrete norms from 
Remark \ref{variables}. More exactly, we have computed 
the relative percentage error for the physical variables of interest, namely 
\begin{equation*}
100\frac{\Delta t\sum_{k=1}^N\norm{\bH(t_k)-\bH_h^k}_{0,\Omega}^2}{\Delta t\sum_{k=1}^N\norm{\bH(t_k)}_{0,\Omega}^2}\qquad 
100\frac{\Delta t\sum_{k=1}^N\norm{\bE(t_k)-\bE_h^k}_{0,\oc}^2}{\Delta t\sum_{k=1}^N\norm{\bE(t_k)}_{0,\oc}^2}
\end{equation*}  
which are time-discrete forms of the errors in $\L^2(0,T;\L^2(\Omega))$ and  $\L^2(0,T;\L^2(\oc))$ norms, respectively.
\\
To show the linear convergence with respect to the mesh-size and the time step, we have computed 
the relative percentage error for the physical variables to $\frac{h}{n}$, $\frac{\Delta t}{n}$, $n=2,\cdots,7$. 
Figure  \ref{ErroresRel_HE} shows log-log plots the magnetic field
and electric field in the conductor domain in the discrete norms considered  above versus
the number of degrees of freedom (d.o.f).
}
\begin{figure}[ht!]
	\begin{center}
		\includegraphics*[width=7cm]{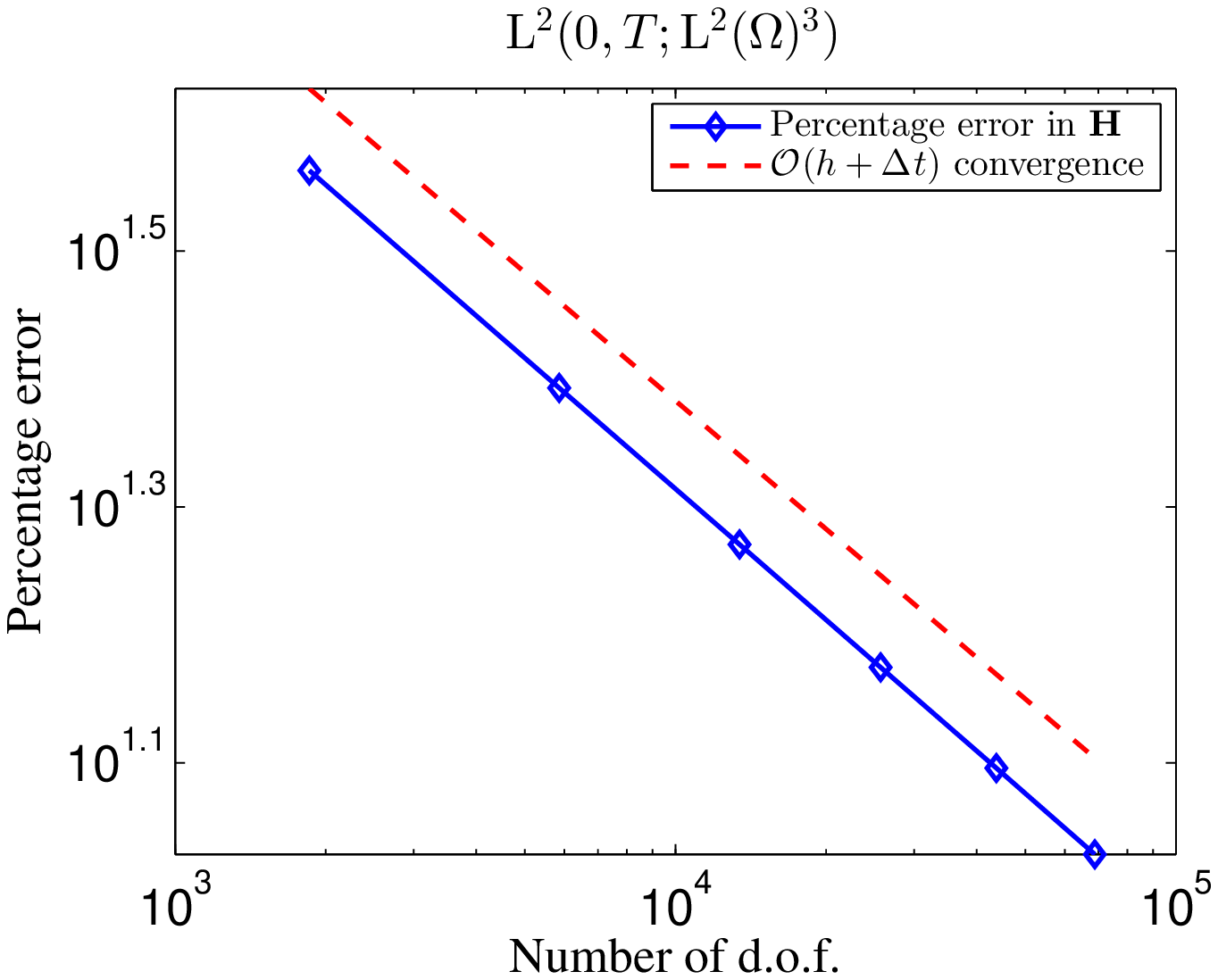}\includegraphics*[width=7cm]{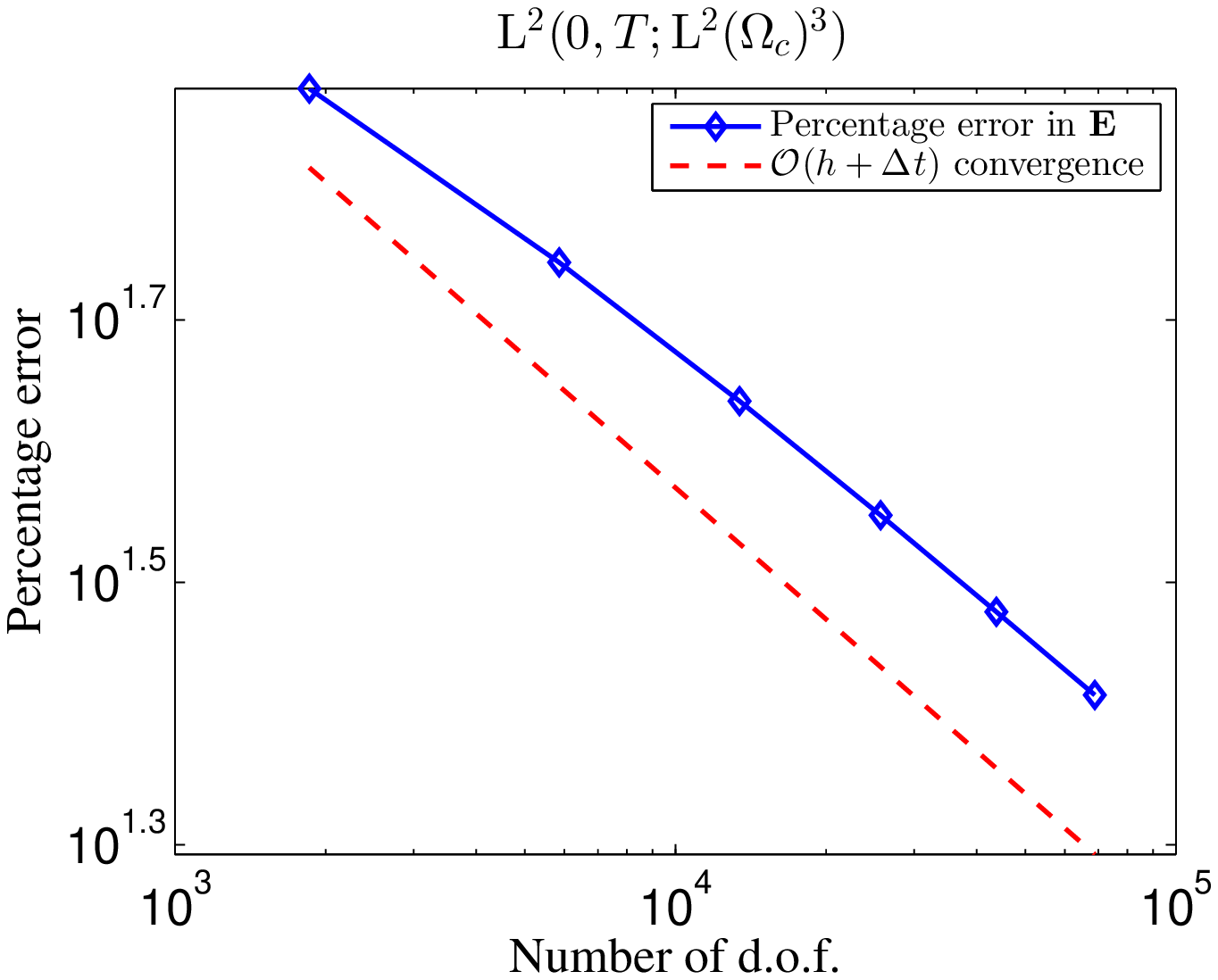}
		\caption{Percentage discretization error curves for $\bH$ (left) and $\bE$ (right)	versus number of d.o.f. (log-log scale).}
		\label{ErroresRel_HE}
	\end{center}
\end{figure}

{Test 2: \textit{A comparison with axisymmetric problem.}} 

\begin{figure}[!ht]
	\begin{center}
		\begin{minipage}{6.cm}
			\includegraphics*[width=5cm]{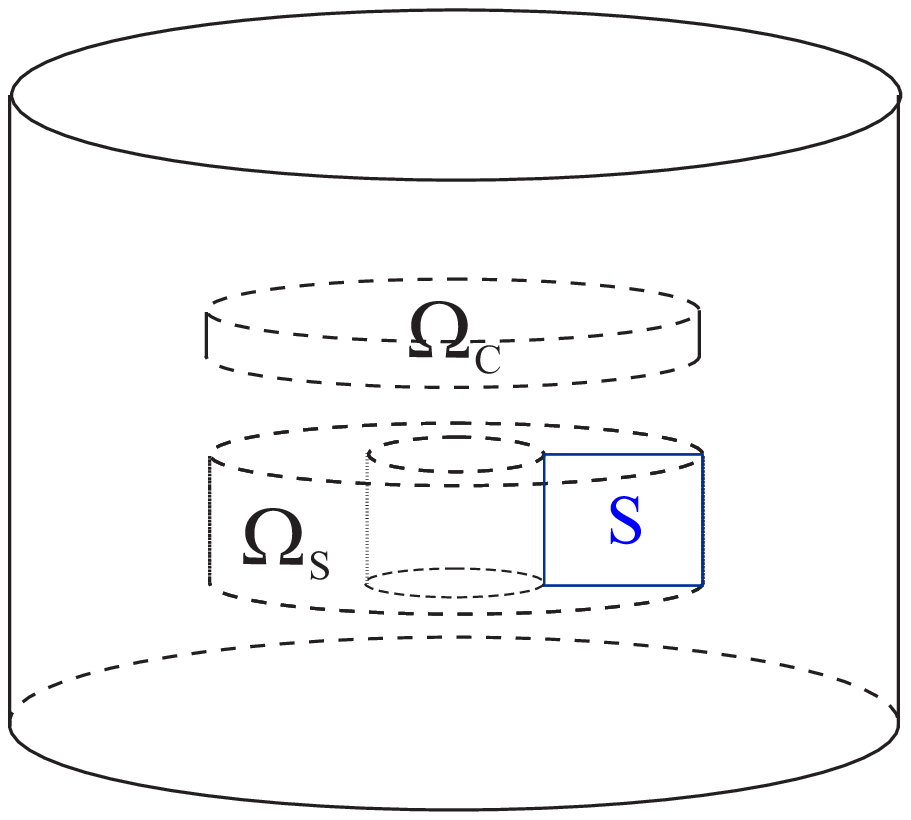}
		\end{minipage}
		\hspace*{1.5cm}
		\begin{minipage}{6.cm}
			\includegraphics*[width=6cm]{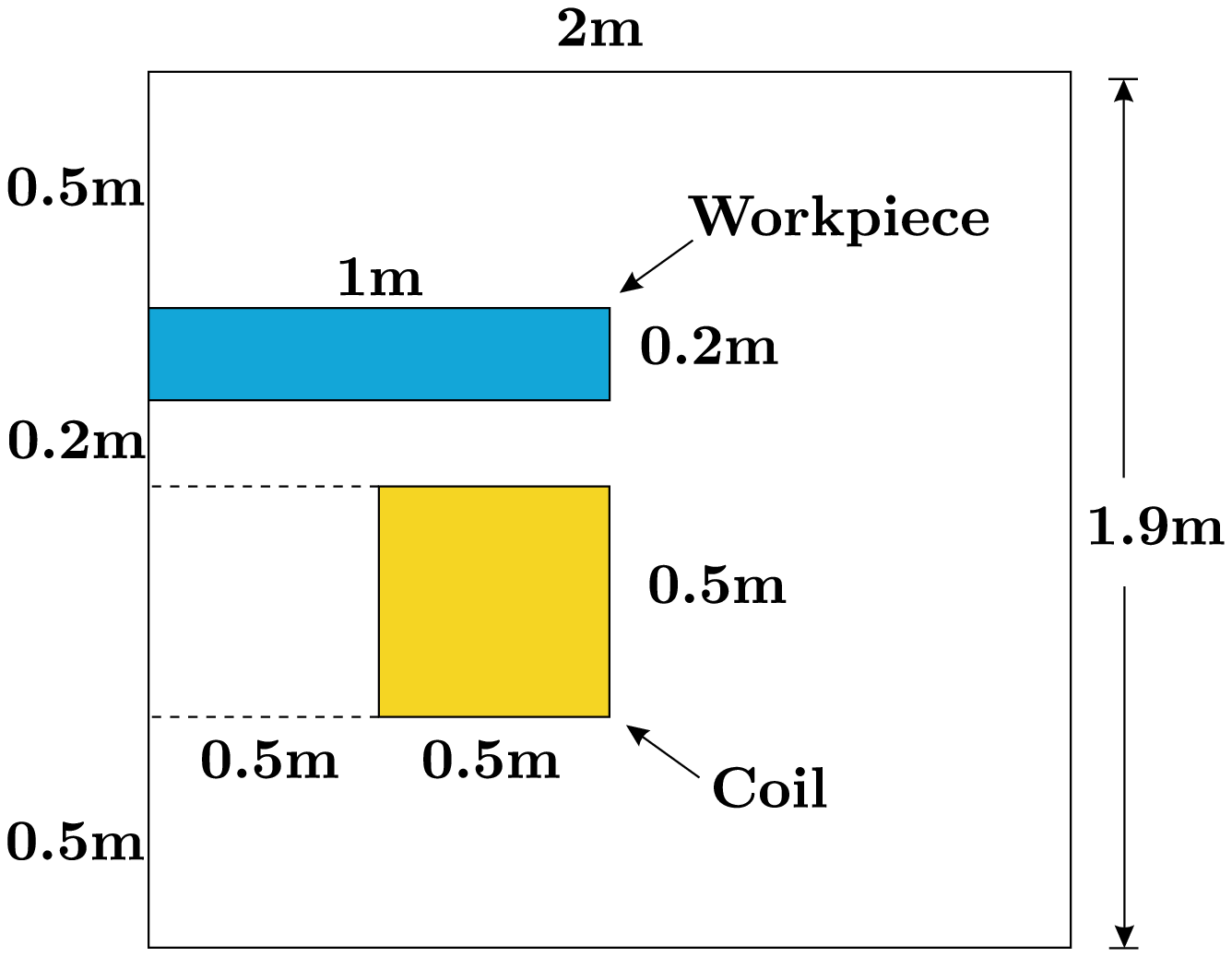}
		\end{minipage}
	\end{center}
	\caption{Sketch of the domain $\Omega$ (left) and its meridian section (right).}
	\label{Ejm3Axis}
\end{figure}

We consider the geometry sketched in Figure~\ref{Ejm3Axis}, which corresponds to a typical EMF (electromagnetic forming) setting. Thus $\oc$ is the cylinder of radius 
$\unit[R_{\rm P}=1]{m}$ and its $z$-coordinate varies between $(1.2,1.4)$.  We assume that $\bJ$ is supported in $\OS$ where $\OS\subset\Omega$, $\OS$ is a toroidal core of rectangular cross section $S$, with inner radius equal to $\unit[R_{\rm s}=0.5]{m}$, outer radius $\unit[R_{\rm S}=1]{m}$ and height $\unit[A_{\rm s}=0.5]{m}$. The source current density is supported in $\OS$ and given by
\begin{equation*}
\bJ(t,\bx)=\frac{I(t)}{\mbox{meas}(S)}
\begin{pmatrix}
-\frac{x_2}{\sqrt{x_1^2+x_2^2}}
\\
\frac{x_1}{\sqrt{x_1^2+x_2^2}}
\\
0
\end{pmatrix}\quad\mbox{in }\OS ,
\end{equation*}
where the current intensity $I(t)$ is shown in Figure~\ref{intensityinputcap4}. Note that, since the source current density field has only azimuthal non-zero component, the solution will be axisymmetric. In particular, we can solve the problem in the meridional section depicted in Figure~\ref{Ejm3Axis} (right). In this case, there is no analytical solution, 
so we will asses the behavior of the method by comparing the computed results with those obtained with an axisymmetric code on the very fine mesh shown in Figure~\ref{mesh} (right) 
which will be taken as `exact' solution.
\begin{figure}[ht]
	\begin{center}
		\includegraphics*[width=6cm]{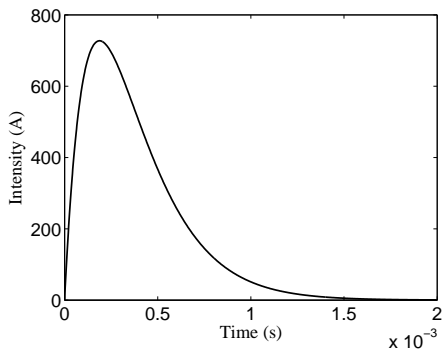}
		\caption{Source current intensity (A) vs. time (s).}
		\label{intensityinputcap4}
	\end{center}
\end{figure}

The axisymmetric problem has been solved by using a scalar formulation written in terms of the azimuthal component of a magnetic vector potential $A_{\theta}$. The corresponding weak formulation, although with boundary conditions different from those of our case, has been analyzed in \cite{BMRRS_SIAM}, \cite{BMRRS_ACM} with moving domains. In particular, the method was proved to converge with optimal order error estimates in terms of $h$ and $\Delta t$ under appropriate assumptions. 
\begin{figure}[!ht]
	\begin{center}
		\begin{minipage}{6.cm}
			\includegraphics*[width=4cm]{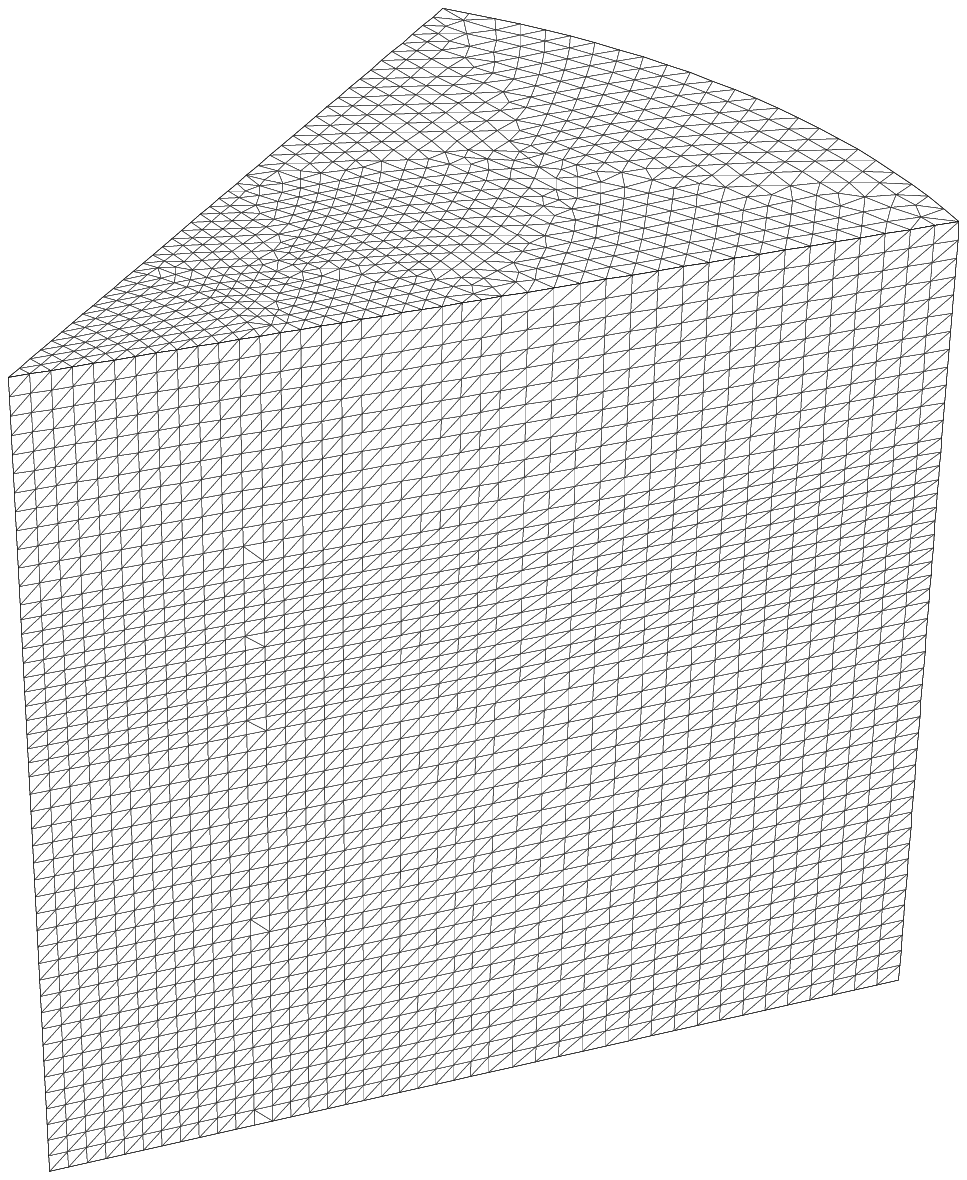}
		\end{minipage}
		\begin{minipage}{6.cm}
			\includegraphics*[width=4cm]{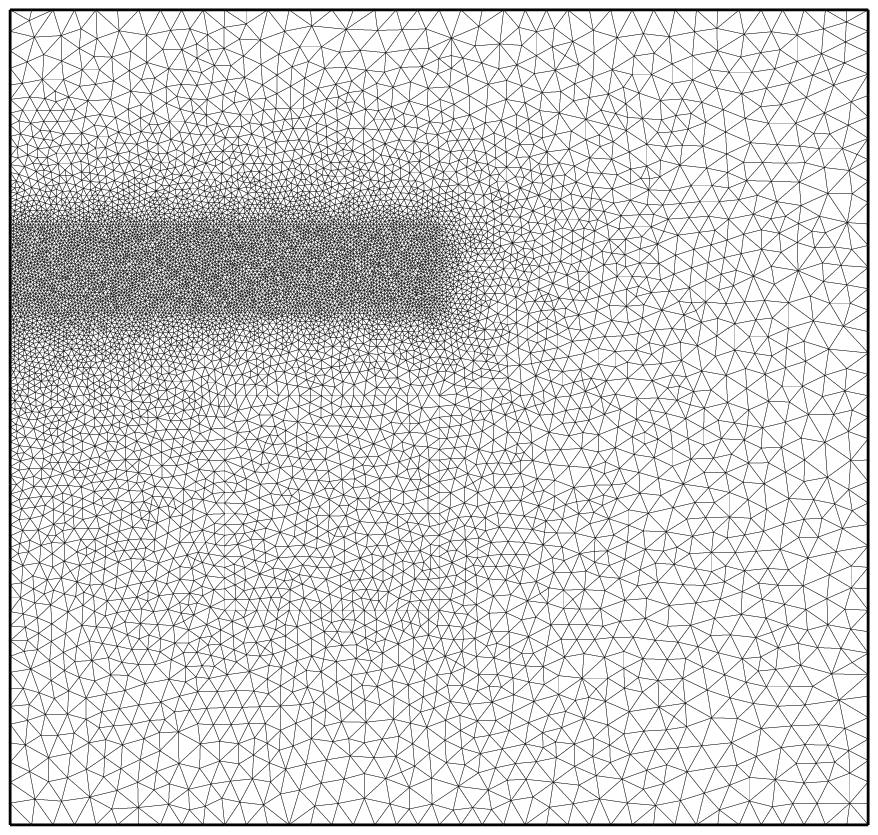}
		\end{minipage}
	\end{center}
	\caption{Coarsest mesh used for the 3D code (left) and mesh used by the axisymmetric code (right).}
	\label{mesh}
\end{figure}

Figure~\ref{mesh} (right) shows the mesh used in the axisymmetric code. Concerning the 3D mesh, we have exploited the symmetry of the problem and solved it in $1/8$ of the whole domain to reduce the number of degrees of freedom. The used mesh is shown in Figure~\ref{mesh} (left).

We have solved the problem with several successively refined meshes and a time-step conveniently reduced to analyze the convergence with respect to both, the mesh-size and the time-step simultaneously. Figure~\ref{ErroresAxi} shows a log-log plot of the relative error for the electric field, versus the number of degrees of freedom (d.o.f.). The curve shows that the obtained results converge to the `exact' ones as $h$ and $\Delta t$ go to zero.
\begin{figure}[ht!]
	\begin{center}
		\includegraphics*[width=7cm]{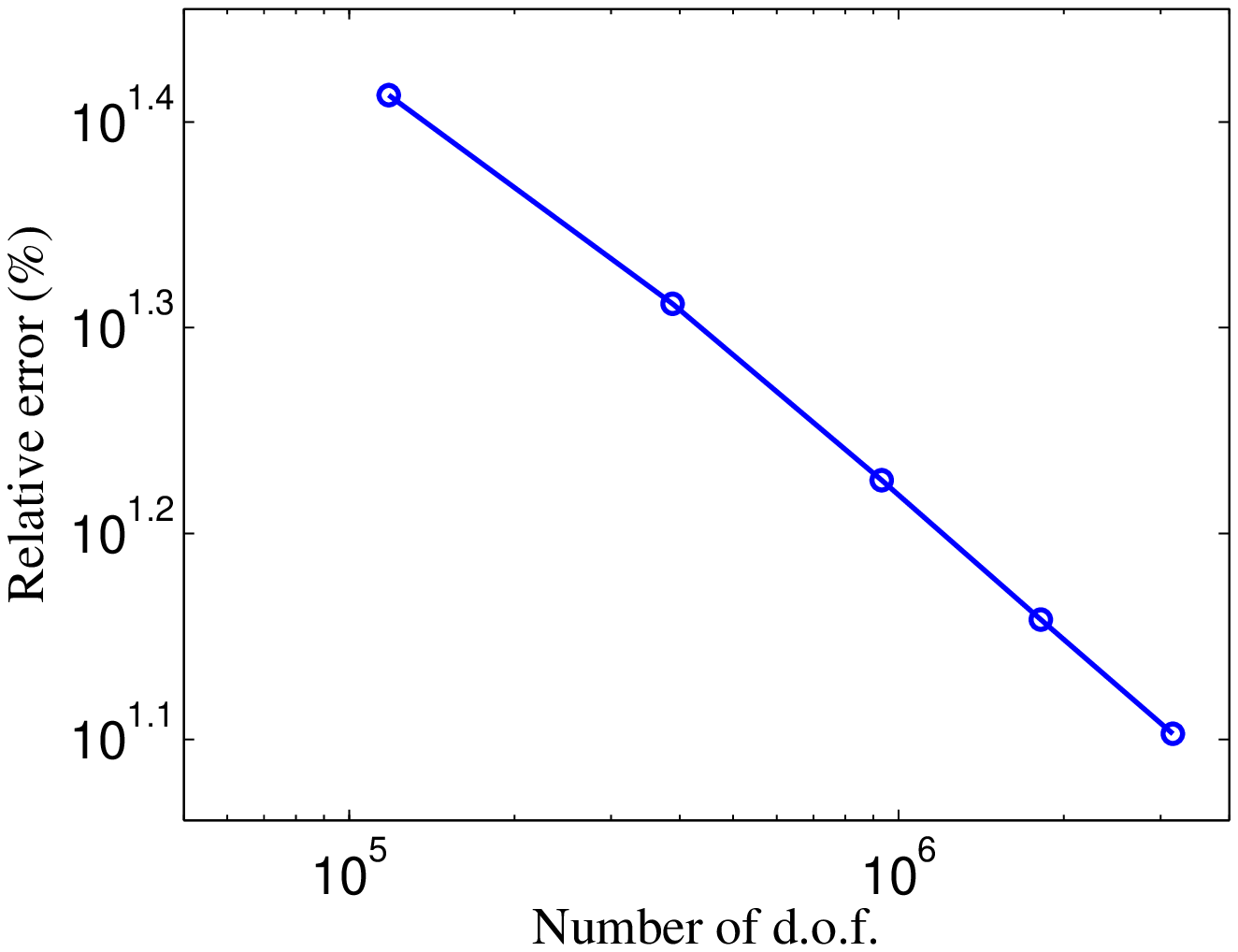}
		\caption{$100\times\frac{\underset{1\leq k\leq M}{\max}
				\|\bE(t_k)-\bE^k_h\|_{\L^2(\oc)^3}} {\underset{1\leq k\leq M}{\max}
				\|\bE(t_k)\|_{\L^2(\oc)^3}}$
			versus number of d.o.f. (log-log scale).}
		\label{ErroresAxi}
	\end{center}
\end{figure}
 
\subsubsection{Input current model}
{As another application of our theoretical framework, we present the fully-discrete analysis for the input current model proposed  in \cite{blrs:13}.
 }To this aim, we introduce the following Hilbert spaces
\begin{align}\label{defXMI}
\X&:=\set{\wn\in\hcurlo:\ \wn\times\nn=\cero\textrm{ on }\Gc,\ \curl\wn\cdot\nn=\cero\textrm{ on }\partial\Omega},\\
\M&:=\set{\varphi\in\hunod:\ \varphi\vert_{\GI^1}=0;\ \varphi\vert_{\GI^k}=C_{I},\ k=2,\ldots,M_{I}}.
\end{align}
with their usual norms in $\hcurlo$ and $\H^1(\od)$, respectively. 

Let $g\in\L^2(0,T;\L^2(\Gamma))$, $I_n\in\H^2(0,T),\,\,n=1,\cdots,N$ and $\Hn_0\in\hcurlo$ the initial magnetic field. We denote 
\begin{equation*}
\left\langle \fn(t),\vn\right\rangle=\int_{\Omega}\fn(t)\cdot\vn :=\sum_{n=1}^N L_n(\vn)(I_n(t)-I_n(0))
+\int_{\Omega}\curl \Hn_0\cdot\vn
\quad \forall \vn\in \X,
\end{equation*}
for any $t\in[0,T]$.
The variational formulation for the  input current model  (see \cite{blrs:13}) is :
\begin{problem}\label{mix1-puertos}
 Find $\un\in\L^2(0,T;\X)$ and $\lambda\in\H^1(0,T;\M)$ such that 
\begin{align*}
&\dfrac{d}{dt}\left[\int_{\oc}\sigma\un(t)\cdot\vn + \int_{\od}\varepsilon\vn\cdot\nabla\lambda(t)\right]
+ \int_{\Omega}\frac{1}{\mu}\curl\un(t)\cdot\curl\vn=\langle \fn(t),\vn\rangle
&& \forall \vn\in \X,
 \\
& \int_{\od}\varepsilon\un(t)\cdot\nabla\mu = \int_{\Gamma_d}\left(\int_0^tg(s)ds\right)\mu&&\forall \mu\in\M,
\\
&\un(\cdot,0)=\cero\quad \textrm{in }\oc \quad \textrm{and } \quad\lambda(0)=0\quad \textrm{in }\od,
\end{align*}
\end{problem} 
where we have introduced the time primitive 
\begin{align*}
\lambda(\xn,t)=\int_0^t\xi(\xn,s)ds\quad \xn\in \od\,,\quad t\in [0,T].
\end{align*}
of the original lagrange multiplier $\xi$ of the model in \cite{blrs:13}.
As before, to get the fully-discrete approximation for Problem \ref{mix1-puertos}, it is necessary to employ $X_h$ and $M_h$, 
finite-dimensional subspaces of $\X$ and $\M$, respectively. Thus, we define the following spaces
\begin{align*}
&\X_h:=\left\{\wn\in \mathcal{N}_h(\Omega):\,\, \wn\times\nn=\cero\,\,\text{on}\,\,\Gc\,\,\text{and}\,\,\curl\wn\cdot\nn=0\,\,\text{on}\,\,\partial{\od}\right\},\\
&\M_h:=\left\{\mu\in\mathcal{L}(\od):\,\,\mu|_{\GI^1}=0,\,\,\mu|_{\GI^k}=C_{k},\,\,\, k=2,\cdots,M \right\},
\end{align*}
where $\mathcal{N}_h(\Omega)$ and $\mathcal{L}(\od)$ are N{\'e}d{\'e}lec (see \eqref{defXh}) and Lagrange finite element spaces, respectively.
\begin{problem}\label{mix1-puertos2}
 Find $\un_h^n\in\X_h$ and $\lambda_h^n\in\M_h$ for $n=1,\cdots,N$  such that 
\begin{align*}
&\left[\int_{\oc}\sigma\frac{\un_h^n-\un_h^{n-1}}{\Delta t}\cdot\vn + \int_{\od}\varepsilon\vn\cdot\frac{\nabla\lambda_h^n-\nabla\lambda_h^{n-1}}{\Delta t}\right]
+ \int_{\Omega}\frac{1}{\mu}\curl\un_h^n\cdot\curl\vn =\int_{\Omega}\fn(t_n)\cdot\vn\quad \forall \vn\in\X_h, \\
 \\
 & \int_{\od}\varepsilon\un_h^n\cdot\nabla\mu= \int_{\Gamma_d}\left(\int_0^{t_n} g(s)ds\right)\mu\quad \forall \mu\in\M_h,
\\
&\un_h^0=\cero\quad \textrm{in }\Omega\quad \textrm{and }\quad  \lambda_h^0=0\quad \textrm{in }\od.
\end{align*}
\end{problem}
Next, we deduce the existence and uniqueness of the  fully-discrete solution  
of Problem \ref{mix1-puertos2}. So, we will prove that the hypotheses H7-H8 hold. Consequently, we define the discrete kernel of $b$ given by
\begin{align*}
\V_h=\left\{\vn\in\X_h:\,b(\vn,\mu)=0\,\,\forall\mu\in\M_h \right\}.
\end{align*}
The proof of discrete \textit{inf-sup} condition H7 is similar to the deduction of its continuous version (see \cite[Theorem 3.2]{AGL}). Then,  we only show the proof of H8. To this end, we need to deduce a discrete version of  \cite[Proposition 7.4]{FG}.
To do that,  we introduce the following  notation:
\begin{align*}
\HH&:=\HGIcurlood\cap\HGddivood,\\
\HHGIcurlod&
:=\set{\wn\in\HGIcurlod:\ \curl\wn\cdot\nn=0 \textrm{ on }\Gd},\\
 \HHdivod & :=\set{\wn\in\hdivod:\ \wn\cdot\nn\vert_{\Gd}\in\LdosGd},\\
\mathcal{N}(\od)&:=\left\{\vn|_{\od}: \vn\in \mathcal{N}(\Omega)\right\},\\
\mathcal{N}(\oc)&:=\left\{\vn|_{\oc}: \vn\in \mathcal{N}(\Omega)\right\},\\
\V_{h,\mathrm D}&:=\left\{\wn\in \HHGIcurlod\cap\HH \cap \mathcal{N}(\od):\,\, b(\wn,\varphi)=0\quad\forall\varphi\in\M_{h}\right\}.
\end{align*}
\begin{lemma}\label{poncairediscreto}
 There exist a constant $C>0$ independent of $h$ such that
\[
\|\vn\|_{0,\od}\leq C\|\curl\vn\|_{0,\od}\quad \forall\vn\in \V_{h,\mathrm{D}}.
\]
\end{lemma}
\begin{proof}
{The proof is adapted from \cite[Lemma 4.7]{alonsovallilibro}. The authors have done the case on the which
the conductors do not go through the boundary of $\Omega$.}
Let  $\vn\in\V_{h,\mathrm{D}}$.  In virtue of an orthogonal decomposition of $\L^2(\od)^3$ (see {\cite[Proposition 6.4]{FG})} we can write 
$\vn=\curl\mathbf{Q}+\grad{\chi}+\mathbf{k}$
with
\begin{align*}
\mathbf{Q}\in \HGdcurlood\cap\HGdcurloodGI \cap \HH^{\perp},\quad \chi\in \hunodGI\quad \text{and} \quad \mathbf{k}\in \HH.
\end{align*}
By substituting $\un=\curl\mathbf{Q}$, it is easy verify
the following $\curl\un=\curl\vn$ in $\od$,  $\dive\un=0$ in $\od$, $\un\cdot\nn=0$ on ${\Gd}$
and $\un\times\nn=\cero$ on ${\GI}$. Then, 
\[
\un\in \HGIcurlod\cap\HGddivood,
\]
and consequently $\un\in\H^s(\od)$ for some $s>1/2$ and there exists $C>0$ such that
\[\|u\|_{s,\od}\leq C\|u\|_{\hcurlod}.\] 
Moreover, since $\un\in \HH^{\perp}$ and by using 
\cite[Proposition 7.4]{FG},  we have
\begin{align}\label{fgd}
\|\un\|_{s,\od}\leq C \|\curl\un\|_{0,\od}.
\end{align}
Furthermore, thanks to that $\curl\un=\curl\vn\in \L^{\infty}(\od)^3$ then we can define $\Pi_h\un\in \mathcal{N}_h(\od)$.
Note that there exists $\phi_h\in\M_h$ such that  
$\grad\phi_h=\Pi_h(\grad\chi+\mathbf{k})$, and 
hence
\begin{align}\label{des-aux1}
\|\vn\|_{0,\od}^2\leq\|\vn\|_{0,\od}\|\Pi_h\un\|_{0,\od}.
\end{align}
On the other hand, for all $K\in \mathcal{T}_h$ con $K\subset\od$, we obtain
\begin{align*}
\|\Pi\un-\un\|_{0,K}&\leq Ch^s\left(\|\un\|_{s,K}+\|\curl\vn\|_{s,K}\right)\\
&\leq C\left( h^s\|\un\|_{s,K}+\|\curl\vn\|_{0,K}\right),
\end{align*}
where we have used the local inverse estimate
\begin{align*}
\|\curl \vn\|_{s,K}\leq C h^{-s}\|\curl\vn\|_{0,K}.
\end{align*}
By using \eqref{fgd} and  triangular inequality, we have
\begin{align*}
\|\Pi_h\un\|_{0,\od}\leq C\|\curl\vn\|_{0,\od}
\end{align*}
Finally, the Lemma follows from \eqref{des-aux1}.
\end{proof}
\begin{lemma}\label{lemaEh}
 If we define
\[
 \X_h(\Omega_c):=\left\{\vn\in\mathcal{N}(\Omega_c):\,\, \vn\times\nn=\cero\quad \text{on}\quad\Gc\right\}.
 \]
 Then, the lineal mapping  $\E_h:\X_{\Gc,h}(\oc)\to\V_{h}$  characterized by 
\begin{equation}\nonumber
\begin{split}
&(\E_h\vnc)\vert_{\oc}=\vnc\qquad\forall\vnc\in\X_{h}(\oc),\\
&\int_{\od}(\curl\E_h\vnc)\cdot\curl\wnd = 0 \qquad\forall \vnc\in\X_{h}(\oc)
\quad \forall\wnd\in\V_{h,d}.
\end{split}
\end{equation}
is well defined and bounded.
\end{lemma}

\begin{proof}
Let us denote $\gtc:\mathcal{N}_h(\oc)\to\rangoGtoc $ and $\gtd:\mathcal{N}_h(\od)\to\rangoGtod$
by tangential traces on $\hcurloc$ and $\hcurlod$, respectively. 
It follows that linear operator $\etan:\X_{h}(\oc)\to\rangoGtod$ given by 
\begin{equation*}
\etan(\vc):=
\left\{
\begin{array}{ll}
\gtc(\vc)\vert_{\GI}&\textrm{on }\GI,\\
\cero,&\textrm{on }\Gd,
\end{array}
\right.
\end{equation*}
is well defined. Moreover, we have
\begin{equation*}
\norm{\etan(\vc)}_{\rangoGtod} = \norm{\gtc(\vc)}_{\rangoGtoc}\leq C_1\norm{\vc}_{\hcurloc}
\qquad\forall\vc\in\HGccurloc.
\end{equation*}

After considering the continuous right inverse of tangential operator $\gtd$, we can
define the continuous linear operator $\LL_h:\X_{h}(\oc)\to\mathcal{N}(\od)$ given by 
\[
\LL_h(\vc):=(\gtd)^{-1}(\etan(\vc))\quad \forall\vc\in \X_h(\oc)
\]
there holds
\begin{equation}\nonumber
\LL_h(\vnc)\times\nn=\cero\quad \text{on}\quad \Gd\quad\text{and}\quad \LL_h(\vnc)\vert_{\GI}\times\nn=\vnc\vert_{\GI}\times\nn. 
\end{equation}
By denoting $\H_h(\od):=\HHGIcurlod\cap\HH^\bot\cap\mathcal{N}_h(\od)$, we consider the following mixed problem 
\begin{problem}\label{mixto-puertosh}
\noindent Find $\znd\in{\H}_h(\od)$ and $\rho\in{M}_h$ such that 
\begin{equation*}
\begin{split}
\int_{\od}\curl\znd\cdot\curl\wnd + b(\wnd,\rho) &= - \int_{\od}\curl(\LL_h(\vnc))\cdot\curl\wnd\qquad
\forall\wn\in \H_h(\od)\\
b(\znd,\mu) &= -b(\LL_h(\vnc),\mu)\qquad\forall\mu\in {M}_h(\od).
\end{split}
\end{equation*}
\end{problem}
Now, we proceed to show that the previous problem is well-posedness. 
 By using the Lemma \ref{poncairediscreto} 
the bilinear form given by
\begin{align*}
(\vnd,\wnd)\mapsto\int_{\od}\curl\vnd\cdot\curl\wnd,
\end{align*}
is coercive on  $\V_{h,d}$. Furthermore, the discrete \textit{inf-sup} conditions is
satisfied. In fact: by noting  $\textbf{grad}(\M_h)\subset{\H}_h(\od)$. Thus, we obtain
\begin{equation}\nonumber
\begin{split}
\sup_{\vnd\in{\H}_h(\od)}\dfrac{b(\vnd,\mu)}{\norm{\vnd}_{\hcurlod}}
&\geq \dfrac{b(\grad\mu,\mu)}{\norm{\grad\mu}_{\hcurlod}}
= \varepsilon_{0}\norm{\grad\mu}_{\ldosod^3}\quad\forall\mu\in M_h.
\end{split}
\end{equation}
It follows from the  Babuska-Brezzi theory  that the Problem \ref{mixto-puertosh}
has a unique solution, which satisfies 
\begin{equation}\nonumber
\norm{\znd}_{\hcurlod}\leq C\norm{\vc}_{\hcurloc}\qquad\forall\vnc\in \X_{h}(\oc).
\end{equation}
 Hence, we define 
\begin{equation}\nonumber
\E_h\vnc:=\left\{
\begin{array}{ll}
\vnc&\textrm{in }\oc,\\
\znd+\LL_h\vnc&\textrm{in }\od,
\end{array}
\right.
\end{equation}
there holds
\[
\vnc\vert_{\GI}\times\nn = \LL_h(\vnc)\vert_{\GI}\times\nn\quad\text{and}\quad \znd\vert_{\GI}\times\nn=\cero,
\]
from which the result follows. 
\end{proof}
The following result may be proved in much the same way as Lemma \ref{dgarding}. 
This is due to Lemma \ref{poncairediscreto} and Lemma \ref{lemaEh}.
\begin{lemma}
There exist positive constants $\widehat\gamma$ and $\widehat\alpha$ such that
\begin{equation}\label{dgarding}
\int_{\Omega}\dfrac{1}{\mu}\vert\curl\vn\vert^2 + \widehat\gamma\int_{\oc}\sigma\vert\vn\vert^2 
\geq \widehat\alpha\norm{\vn}_{\hcurlo}^2\qquad\forall\vn\in V_h.
\end{equation}
\end{lemma}

\begin{proof}
Let $\vn\in\V_h$, considering $\vnc:=\vn\vert_{\oc}$ and $\E_h\vnc$ given by 
Lemma~\ref{lemaEh}, we can define $\widetilde\wn:=\vn-\E_h\vnc$, then
\[
\widetilde{\wn}=\cero \textrm{ in $\oc$},\qquad\widetilde{\wn}\in\V_h,\qquad\widetilde{\wn}\vert_{\od}\in\V_{h,d}.
\]
Thus, by using the Lemma \ref{poncairediscreto}, the continuity of $\E_h$ and 
proceeding as in Lemma \ref{lemmaGIC}, it is deduced the result.
\end{proof}

Consequently, we can use the results of Section~\ref{discreto} to conclude that the fully-discrete approximation of the 
Problem \ref{mix1-int} has a unique solution $(\un_{h}^n,\lambda_{h}^n)\in X_h\times M_h,\ n=1,\dots,N$.
Now, our next to goal is to obtain error fully-discrete scheme. Before, we recall that $\lambda=0$ (see \cite[Theorem 2.4]{blrs:13}). 
{As in subsection \ref{IC} and by assuming $\un\in \H^1(0,T;\X)\cap\H^2(0,T;\ldoso^3)$ (see \eqref{defXMI}), we can obtain  
similar results as in  Theorem  \ref{cea-eddy} and Theorem \ref{errordu}. 
These results allow us to obtain
the asymptotic error estimates. In fact, fixing an index $r>\frac{1}{2}$
and considering  $\mathcal{X}:=\mathbf{H}^r(\curl,\Omega)\cap \X$ (see \eqref{Hr}),
according to \cite[Lemma 2.2]{BPR3},  the N\'ed\'elec  interpolant operator $I_h^{\mathcal{N}}:\mathcal{X} \to \X_h$ is well defined
and we can easily obtain an analogous result to Corollary~\ref{coroconvucd}. 
Thus, we easily obtain similar error estimates to those that were given in Remark \ref{variables} for the approximation of the electric and magnetic field at each 
time step.
Finally, for some numerical results of this subsection that confirm the theoretical result obtained
in this work, we refer the reader to \cite[Section 4]{blrs:13}.}


\textbf{Funding}\\
This work was partially supported by Colciencias through the 727 call, University of Cauca through project VRI ID 5243 and by Universidad Nacional de Colombia through Hermes project $46332$.

\textbf{Availability of data and materials}\\
Not applicable.

\textbf{Competing interests}\\
The authors declare that they have no competing interests.

\textbf{Authors contributions}\\
The authors declare that the work was realized in collaboration with the same responsibility. All authors read and
approved the final manuscript.




\begin{thebibliography}{}



\bibitem{A1}
\textsc{Acevedo, R.,  Meddahi, S. \& Rodr\'iguez, R.}(2009)
\newblock An $\boldsymbol{E}$-based mixed formulation for a time-dependent
eddy current problem. 
\newblock \emph{Mathematics of Computation}, \textbf{78}, 1929--1949.

\bibitem{A2} 
\textsc{ Acevedo, R. and Meddahi, S.}(2011)
\newblock An $\boldsymbol{E}$-based mixed FEM and BEM coupling for a time-dependent
eddy current problem.
\newblock \emph {IMA Journal of Numerical Analysis}, \textbf{31}(2), 667--697.

\bibitem{AGL} 
\textsc{ Acevedo, R. C. G\'omez and  L{\'{o}}pez-Rodr{\'{\i}}guez, B. }(2020)
\newblock Well-posedness for a family of degenerate parabolic mixed equations. Submitted,
\newblock \emph {arXiv preprint arXiv:3172886}.

\bibitem{A4}
\textsc{Acevedo, R.  Navia, P. and  Alvarez, E.} 
\newblock {A boundary and finite element coupling for a magnetically nonlinear eddy current problem},
\newblock \emph{Electronic Transactions on Numerical Analysis}, \textbf{50},  2020, 230-248. 

\bibitem{AB}
\textsc{Ammari, H.  Buffa, A. and  N\'ed\'elec, A. C.} 
\newblock {A justification of eddy currents model for the Maxwell equation},
\newblock \emph{SIAM J. Appl. Math.}, \textbf{60},  2000, 1805--1823. 




\bibitem{alonsovalli}
\textsc{Alonso Rodr{\'{\i}}guez \& A.  and Valli, A.}(1999)
\newblock{An optimal domain decomposition preconditioner for low-frequency time-harmonic {M}axwell equations}
\newblock \emph{Mathematics of Computation}, \textbf{68}, 607--631.

\bibitem{alonsovallilibro}
\textsc{Alonso-Rodr{\'{\i}}guez, A. \& Valli, A.}(2010)
\newblock {Eddy Current Approximation of Maxwell Equations: Theory, algorithms and applications}, 
\newblock\emph{Springer} 
%
\bibitem{AB}
\textsc{Ammari, H.  Buffa, A. \&  N\'ed\'elec, A. C}(2000)
\newblock {A justification of eddy currents model for the Maxwell equation},
\newblock \emph{SIAM J. Appl. Math.}, \textbf{60}, 1805--1823. 
%
\bibitem{BR}
\textsc{Bernardi, C. \&  Raugel, G.}(1985)
\newblock {A conforming finite element method for the time-dependent {N}avier-{S}tokes equations},
\newblock \emph{SIAM J. Numer. Anal.}, \textbf{22}, 455--473.

\bibitem{BG}
\textsc{Boffi, D. \&  Gastaldi, L.}(2004)
\newblock Analysis of finite element approximation of evolution problems in mixed form,
\newblock\emph{SIAM J. Numer. Anal.}, \textbf{42}, 1502--1526.
\bibitem{amrouche}
\textsc{Amrouche, C. and Bernardi, C. and Dauge, M. \& Girault, V.}(1998)
\newblock {Vector potentials in three-dimensional non-smooth domains},
\newblock \emph{Math. Methods Appl. Sci.},  \textbf{21}, 823--864.

\bibitem{CR}
{\sc J. Cama{\~ n}o, and R. Rodr\'iguez}, {\em Analysis of a FEM-BEM model posed on
the conducting domain for the time-dependent eddy current problem}, {J.
Comput. Appl. Math.}, vol. 236, no. 13, pp. 3084–3100, 2012. 

\bibitem{EJ}
\textsc{Emson, C.R.I. and Simkin, J.},(1983)
\newblock {An optimal method for 3D eddy currents.}.
\newblock{ IEEE Trans. Magn} {\textbf{19}}  2450-2452.


\bibitem{FG}
\textsc{P. Fernandes and G. Gilardi},(1997)
\newblock {Magnetostatic and electrostatic problems in inhomogeneous anisotropic
media with irregular boundary and mixed boundary conditions}.
\newblock{ Math. Models Meth. Appl. Sci.} {\textbf{7}}  957--991.

\bibitem{BPR3}
\textsc{Berm\'{u}dez, A., Rodr\'{\i}guez R.  and Salgado, P.}(2005) 
\newblock{Numerical analysis of electric field formulations of the eddy current
model.}
\newblock \emph{ Numer. Math.}, \textbf{102} 181--201.

\bibitem{BLRS4}
\textsc{A. Berm\'{u}dez, B. L\'opez-Rodr\'iguez, R. Rodr\'{\i}guez and P. Salgado}(2012) ,
\newblock  Numerical analysis of a penalty approach for the solution of a
transient eddy current problem.
\newblock\emph{ Comput. Math Appl.} \textbf{9} 2503-2526. 

\bibitem{BLRS1}
\textsc{A. Berm\'{u}dez, B. L\'opez-Rodr\'iguez, R. Rodr\'{\i}guez and P. Salgado}(2012) ,
\newblock  Numerical solution of transient eddy current problems with input
current intensities as boundary data.
\newblock\emph{ IMA J. Numer. Anal.} \textbf{32} 1001--1029. 



\bibitem{blrs:13}
\textsc{Berm{\'u}dez, A., L{\'{o}}pez-Rodr{\'{\i}}guez, B., Rodr{\'{\i}}guez, R. \& Salgado, P.}(2013)
\newblock {An eddy current problem in terms of a time-primitive of the electric field with non-local source conditions},
\newblock \emph{ESAIM - Mathematical Modelling and Numerical Analysis}, \textbf{47}, 875--902. 

\bibitem{BMRRS_SIAM}
\textsc{Berm\'udez, A., Mu\~noz, R., Reales, C., Rodr\'iguez, R. and Salgado, P.}(2013)
\newblock{A transient eddy current problem on a moving domain. Mathematical analysis},
\newblock\emph{SIAM Journal on Mathematical Analysis}, \textbf{45}, 3629--3650. 

\bibitem{BMRRS_ACM}
\textsc{Berm\'udez, A., Mu\~noz, R., Reales, C., Rodr\'iguez, R. and Salgado, P.}(2016)
\newblock{A transient eddy current problem on a moving domain. Numerical analysis},
\newblock\emph{Advances in Computational Mathematics}, \textbf{42}, 757--789.


\bibitem{brs:02}
\textsc{Berm{\'u}dez, A., Rodr{\'{\i}}guez, R. \& Salgado, P.}(2002)
\newblock {A finite element method with {L}agrange multipliers for low-frequency harmonic {M}axwell equations},
\newblock \emph{SIAM J. Numer. Anal.}, \textbf{40}, 1823--1849. 


\bibitem{bossavit}
\textsc{Bossavit, A.}(1998)
\newblock\emph{Computational Electromagnetism},
\newblock{Academic Press Inc.} 

%


\bibitem{ciarlet}
\textsc{Ciarlet, P.}(2002)
\newblock\emph{The Finite Element Method for Elliptic Problems},
\newblock{SIAM}





\bibitem{guermond}
\textsc{Ern, A. \& Guermond, J.-L.}(2004)
\newblock\emph{Theory and Practice of Finite Elements},
\newblock{Spinger-{V}erlag}
\bibitem{hiptmair}
\textsc{Hiptmair. R}(2002)
\newblock\emph{ Finite elements in computational electromagnetism},
\newblock{Numerica}, 11, 2002, pp.~237--339.

\bibitem{JT}
\textsc{Johnson, C. \&  Thom\'ee, V.}(1981)
\newblock Error estimates for some mixed finite element methods for parabolic type problems.
\newblock \emph{RAIRO Anal. Num\'er.}, \textbf{15}, 41--78.

 \bibitem{panton}
 \textsc{Panton R.~ L.} (2013)
 \newblock\emph{Incompressible Flow},
 \newblock \emph{Wiley}, New York, 2013.

\bibitem{MS1}
\textsc{Meddahi, S. \& Selgas, V.}(2003)
\newblock A mixed-{FEM} and {BEM} coupling for a three-dimensional eddy
              current problem
 \newblock \emph{M2AN Math. Model. Numer. Anal.}, \textbf{37}, 291--318
 
 
\bibitem{mini}
\textsc{Arnold, D.N., Brezzi, F. \& Fortin, M.}(1984)
\newblock {A stable finite element for the stokes equations},
\newblock \emph{Calcolo}, \textbf{21}, 337-344
%



%

%
%
%

\bibitem{WB}
\textsc{ Wei$\ss$ B. and  B\'{\i}r\'o O.} (2004)
\newblock On the convergence of trasient eddy-current problems.
\newblock {\emph{ IEEE Trans. Magn }} \textbf{40}  957--960.



\bibitem{ZCWNM06}
\textsc{Zheng, W., Chen,  Z.  and  Wang, L.} (2006)
\newblock An adaptive finite element method for the $H$-$\psi$ formulation of
time-dependent eddy current problems.
\newblock {\emph{ Numer. Math.}} \textbf{103}  667--689.

\end{thebibliography}
\end{document}